\documentclass[11pt]{amsart} 
\usepackage{amsmath,epsf,amssymb,cite,pifont,amsthm,mathrsfs,epsfig, amsthm,setspace}

\usepackage{hyperref}
\hypersetup
{
    colorlinks,	%
    citecolor=purple,%
    filecolor=black,%
    linkcolor=purple,%
    urlcolor=purple
}

\usepackage{graphicx}
\usepackage{caption}
\usepackage{subcaption}

\usepackage{algorithm}
\usepackage{algpseudocode}

\graphicspath{{Figures/}}

\input xy
\xyoption{all}
\usepackage{Liz}

\makeatletter
\newtheorem*{rep@theorem}{\rep@title}
\newcommand{\newreptheorem}[2]{%
\newenvironment{rep#1}[1]{%
 \def\rep@title{#2 \ref{##1}}%
 \begin{rep@theorem}}%
 {\end{rep@theorem}}}
\makeatother

\newtheorem{thm}{Theorem}[section]
\newreptheorem{theorem}{Theorem}

\newtheorem{lemma}[thm]{Lemma}
\newreptheorem{lemma}{Lemma}

\newtheorem{prop}[thm]{Proposition}

\newtheorem{defn}[thm]{Definition}

\newtheorem{cor}[thm]{Corollary}
\newreptheorem{cor}{Corollary}

\begin{document}
\title{Probabilistic Fr\'{e}chet Means \\for Time Varying Persistence Diagrams\footnote{AMS Subject Classification: 55, 60}}


\author{Elizabeth Munch$^{1}$, Katharine Turner$^{2}$, Paul Bendich$^{2}$,  Sayan Mukherjee$^{4}$,\\ Jonathan Mattingly$^{5}$, and John Harer$^{6}$}
\address{$^1$ \textit{Corresponding author}, Dept of Mathematics \& Statistics, University at Albany -- SUNY, \textit{emunch\at albany.edu}.}
\address{$^2$  Dept of Mathematics,  University of Chicago, \textit{kate\at math.uchicago.edu}.}
\address{$^3$ Dept of Mathematics, Duke University, \textit{bendich\at math.duke.edu}.}
\address{$^4$  Depts of Statistical Science, Computer Science, Mathematics, and Institute for Genome Sciences \& Policy,  Duke University, \textit{sayan\at stat.duke.edu}.}
\address{$^5$ Dept of Mathematics, Duke University, \textit{jonm\at math.duke.edu}.}
\address{$^6$  Depts of Mathematics, Computer Science, and Electrical and Computer Engineering, Duke University, \textit{john.harer@duke.edu}.}
 \date{\today}
\maketitle

\begin{abstract}

In order to use persistence diagrams as a true statistical tool, it would be very useful to have a good notion of mean and variance for a set of diagrams.
In \cite{Mileyko2011}, Mileyko and his collaborators made the first study of the properties of the \Frechet mean in $(\DD_p,W_p)$, the space of persistence diagrams
equipped with the p-th Wasserstein metric. In particular, they showed that the \Frechet mean of a finite set of diagrams always exists, but is not necessarily unique.
The means of a continuously-varying set of diagrams do not themselves (necessarily) vary continuously, which presents obvious problems
when trying to extend the \Frechet mean definition to the realm of vineyards.

We fix this problem by altering the original definition of \Frechet mean so that it now becomes a probability measure on the set of persistence diagrams; in a nutshell, the mean of a set of diagrams will be a weighted sum of atomic measures, where each atom is itself a persistence diagram determined using a  perturbation of the input diagrams.
This definition gives for each $N$ a map $(\DD_p)^N \to \P(\DD_p)$.
We show that this map is H\"older continuous on finite diagrams and thus can be used to build a useful statistic on time-varying persistence diagrams, better known as vineyards.
\end{abstract}

\section{Introduction}

The field of topological data analysis (TDA) was first introduced \cite{Edelsbrunner2000} in 2000, and has rapidly been applied to many different  areas: for example, in the study of protein structure \cite{Agarwal2006,Ban2004,Headd2007}, plant root structure \cite{Galkovskyi2012}, speech patterns \cite{Brown2009},  image compression and segmentation \cite{Carlsson2008, Edelsbrunner2003}, neuroscience \cite{Dabaghian2012}, orthodontia \cite{Gamble2010}, gene expression \cite{Dequeant2008}, and signal analysis \cite{Perea2013}.

A key tool in TDA is the \emph{persistence diagram} \cite{Edelsbrunner2000,Chazal2009b}. 
Given a set of points $S$ in some possibly high-dimensional metric space, the persistence diagram $D(S)$ is a computable summary of the data which provides a compact two-dimensional 
representation of the multi-scale topological information carried by the point cloud; 
see Fig.~\ref{F:circleExample} for an example of such a diagram and Section \ref{sec:DiagVin} for a more rigorous description. 
If the point cloud varies continuously over time (or some other parameter) then the persistence
diagrams vary continuously over time \cite{Cohen-Steiner2007}; the diagrams stacked on top of each other then form what is called a vineyard \cite{Cohen-Steiner2006}.

A key part of data analysis is to model variation in data.
In particular, there is an interest in object oriented data analysis where the data of study is a more complicated object than just points in Euclidean space.
Thus, there has been a recent effort to study the mean and variance of a set of persistence diagrams \cite{Mileyko2011,Turner2011,Blumberg2012,Bubenik2012}, as well as nice convergence rates for persistence diagrams of larger and larger point clouds sampled from a compactly-supported measure \cite{chazal2013}.
There are a variety of reasons to want to characterize statistical properties of diagrams. 
For example, given a massive point cloud $S$, there is a computational and statistical advantage to subsampling the data to produce smaller point clouds $S_1, \ldots, S_n$, and computing the mean and variance of the set of persistence diagrams obtained from the $n$ subsampled data sets. 
In statistical terminology, this example consists of computing a bootstrap estimate \cite{Efron1994} of persistence diagram of the data.
This procedure requires a good definition for the mean (and variance) of a set of persistence diagrams.

The papers \cite{Mileyko2011,Turner2011} make careful study of the geometric and analytic properties of the space $(\DD_p, W_p)$ of persistence diagrams equipped with the Wasserstein metric. This enables defining the mean and variance via the Fr\'echet function\cite{Mileyko2011}, and an algorithm for their computation\cite{Turner2011}.
There are, however, unfortunate problems with using the Fr\'echet mean: the mean of a set of diagrams is not necessarily unique nor continuous.

In this paper, we provide an alternative definition for the mean of a set of diagrams which we call the Probabilistic \Frechet Mean (PFM).
By combining the notions of the Fr\'echet mean and the trembling hand equilibrium in game theory \cite{selten1975reexamination}, we construct a mean that is not itself a diagram, but is rather a probabilistic mixture of diagrams and
thus an element of $\PP(\DD_p)$, the space of probability distributions over persistence diagrams. 
Uniqueness of this new mean will be obvious from the definition we propose.
More crucially, we prove the following corollary to the main technical result, Thm.~\ref{Thm:Main}.

\begin{repcor}{Cor:MapCont}
 Let 
 \begin{center}
 \begin{tabular}{rccc}
  $\Phi:$& $(S_{M,K})^N$  &$\longrightarrow$&  $\PP(S_{M,NK})$\\
        & $(X_1,\cdots,X_N)$ &$\longmapsto$ & $\mu_{X}$
 \end{tabular}
\end{center}
be the map which sends a set of diagrams to its PFM.
Then $\Phi$ is \Holder continuous with exponent $1/2$. 
That is, there is a constant $C'$ such that
\begin{equation*}
\WW_2(\mu_X,\mu_Y)\leq C'\sqrt{{\overrightarrow{d_2}(X,Y)}}
\end{equation*}
for all $X,Y \in (S_{M,K})^N$. 
 Here $\WW_2$ is the Wasserstein metric on the space of probability distributions over the space of persistence diagrams, and $\overrightarrow{d_2}$ is the metric on $(S_{M,K})^N$ induced from using the Wassertstein analogous metric on the space of persistence diagrams coordinate-wise.
\end{repcor}

Thus, if we compute the new mean on each step of a path in $\DD_p$, the resulting object gives a path in $\PP(\DD_p)$, thus making the construction amenable to analyzing distributions of vineyards.
This is stated specifically in the following corollary.

\begin{repcor}{Cor:VineCont}
 Let $\gamma_1,\cdots,\gamma_N:[0,1]\to \DD_2$ be vineyards in $\VV_2$.
 Then 
 \begin{equation*}
  \begin{array}{rccc}
   \mu_\gamma:&[0,1] &\longrightarrow& \PP(S_{M,K})\\
   & t & \longmapsto & \mu_{\gamma_1(t),\cdots,\gamma_N(t)}   
  \end{array}
 \end{equation*}
 is continuous.
\end{repcor}

Finally, we give examples of this mean computed on diagrams drawn from samples of various point clouds, and introduce a useful way to visualize them.

\paragraph{{\bf Outline}}
Section \ref{sec:DiagVin} contains definitions for persistence diagrams and vineyards, as well as a discussion of the space $(\DD_p, W_p)$.
The contributions of \cite{Mileyko2011} and \cite{Turner2011} are reviewed more fully in Section \ref{sec:FMD}, and
the non-uniqueness issue is also discussed in that section.
We give our new definition, the probabilistic \Frechet mean (PFM), in Section \ref{sec:MeanDist}, and prove its desirable theoretical properties in Section \ref{sec:Theorems}.
Examples, implementation details, and a discussion of visualization are in Section \ref{sec:Examples},
and the paper concludes with some discussion in Section \ref{sec:Conclusion}.

\section{Diagrams and Vineyards}
\label{sec:DiagVin}

Here we give the basic definitions for persistence diagrams and vineyards, and then move on to a description of the metric space $(\DD_p, W_p)$.
For more details on persistence, see \cite{Edelsbrunner2008}.
We assume the reader is familiar with homology;  \cite{Munkres2} is a good reference.
We note that all homology groups in this paper are computed with field coefficients.

\subsection{Persistent Homology}

To define persistent homology, we start with a nested sequence of topological spaces,
\begin{equation}  \label{E2:Filtration}
 \emptyset = \X_0 \subseteq \X_1 \subseteq \X_2 \subseteq \cdots \subseteq \X_n = \X.
\end{equation}
Often this sequence arises from the sublevel sets of a continuous function, $f: \X \to \R$, 
where $\X_i = f \inv ((-\infty, a_i])$ with $a_0 \leq a_1\leq \cdots \leq a_n$.
For many applications, this function is the distance function 
\begin{equation*}
 d_S(x) = \inf_{v \in S} \|x-v\|
\end{equation*}
from a point cloud $S$ such as in the example of Fig.~\ref{F:circleExample}.
In this case, a sublevel set can be visualized as a union of balls around the points in $S$.

\begin{figure}[tb]
 \centering
  \includegraphics[width = .4\textwidth]{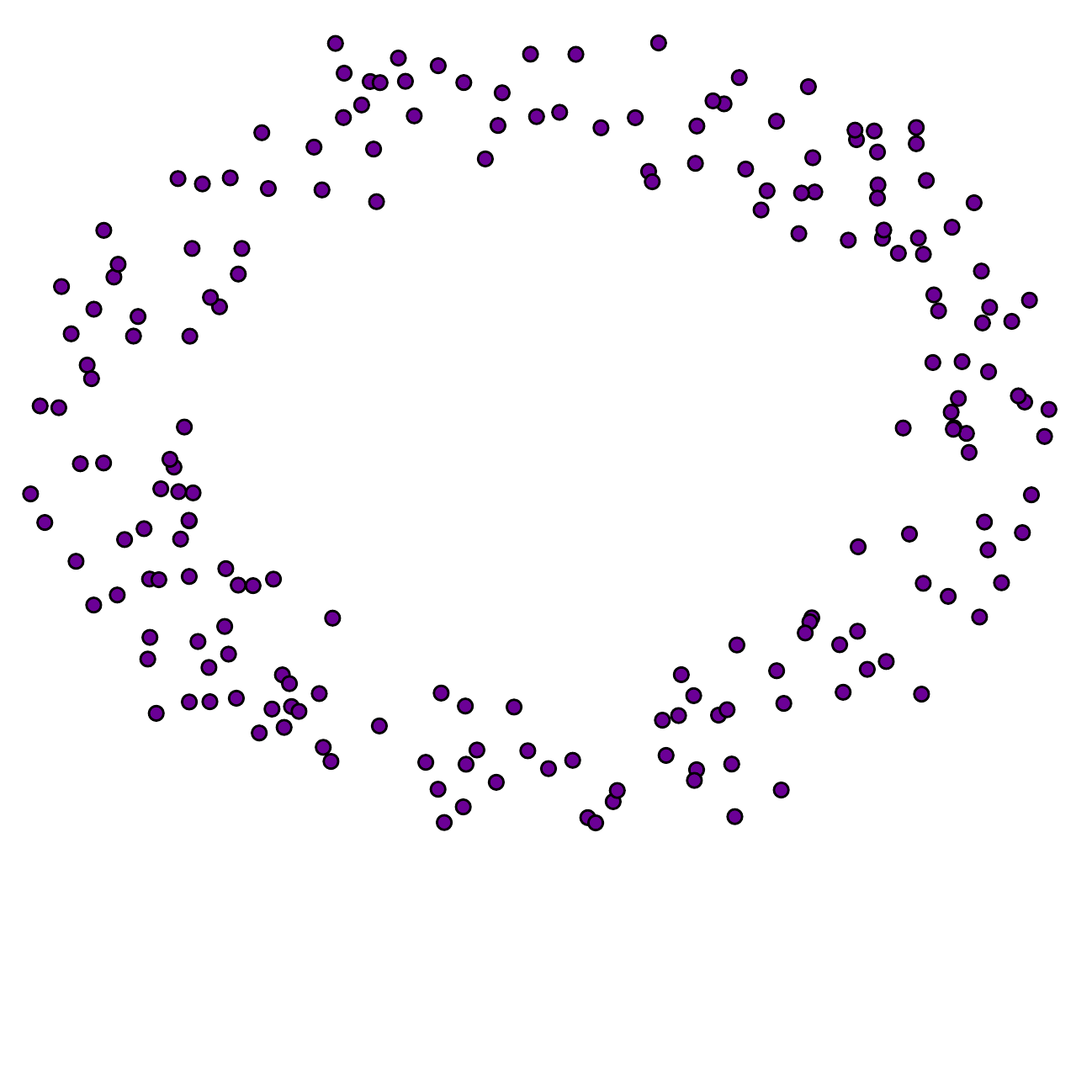}
 \qquad
  \includegraphics[width = .4\textwidth]{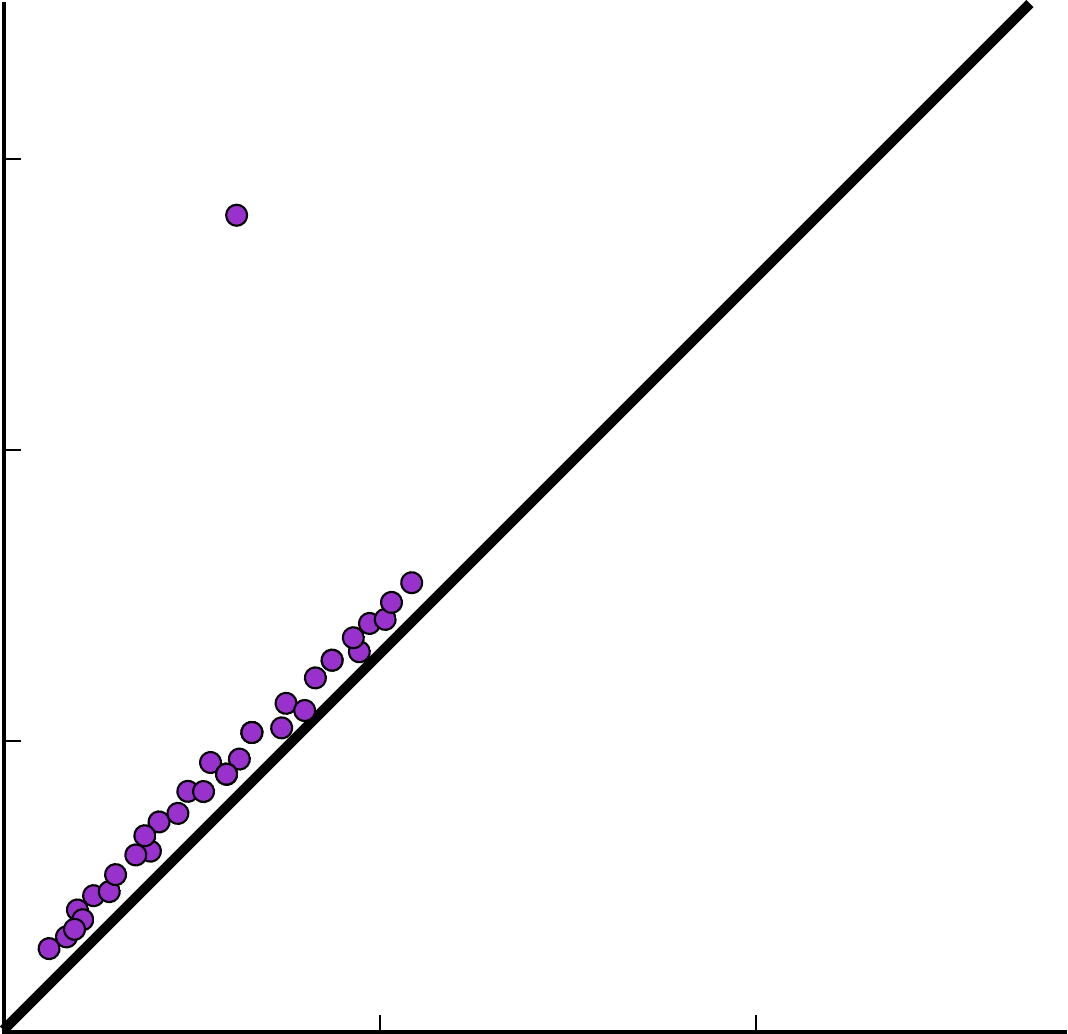}
 \caption[Noisy circular point cloud example]{A point cloud, shown at left, is sampled from an annulus.  In order to summarize the topological data, we look at the sublevel sets of the distance function from the set of points, then construct the persistence diagram, shown at right.  The points near the diagonal are considered noise, while the single point far from the diagonal gives information about the hole in the annulus. }
\label{F:circleExample}
\end{figure}

The sequence of inclusion maps from Eqn.~\eqref{E2:Filtration} induces maps on homology for any dimension $r$,
\begin{equation}\label{E2:Homology}
 \xymatrix{
 0 \ar[r] & H_r(\X_1) \ar[r] & H_r(\X_2) \ar[r] & \cdots \ar[r] & H_r(\X_n).
 }
\end{equation}
In order to understand the changing space, we look at where homology classes appear and disappear in this sequence.  

Let $\phi_i^j: H_r(\X_i) \longrightarrow H_r(\X_j)$ be the composition of the appropriate maps from Eqn.~\eqref{E2:Homology}.
The homology class $\gamma \in H_r(\X_i)$ is said to be born at $\X_i$ if it is not in the image of $\phi_{i-1}^i$.
This same class is said to die at $\X_j$ if its image in $H_r(\X_{j-1})$ is not in the image of $\phi_{i-1}^{j-1}$, but its image in $H_r(\X_j)$ is in the image of $\phi_{i-1}^{j}$. 
In the case that the spaces arose from the level sets of a function $f$ as defined above, we define the persistence of a class $\gamma$ which is born at $\X_i = f\inv((-\infty,a_i])$ and dies at $\X_j=f\inv((-\infty,a_j]$ to be $\pers(\gamma) = a_j-a_i$.  


Notice that this equivalence can also be seen from working with persistence modules \cite{Chazal2009b}, an abstraction of the definition presented here where persistence is defined at the algebraic level.
In fact, given any set of maps between vector spaces, 
\begin{equation*}
 \xymatrix{
  V_1 \ar[r] & V_2 \ar[r] & \cdots \ar[r] & V_n,
 }
\end{equation*}
we can analogously define the birth and death of classes in the vector spaces.

In order to visualize the changing homology, we draw a persistence diagram $d_r$ for each dimension $r$.  
A persistence diagram is a set of points with multiplicity in the upper half plane $\{ (b,d) \in \R^2 \mid d\geq b\}$ along with countably infinite copies of the points on the diagonal $\Delta = \{(x,x) \in \R^2\}$.
For each class $\gamma$ which is born at $\X_i$ and dies at $\X_j$, we draw a point at $(a_i,a_j)$.
A point in the persistence diagram which is close to the diagonal represents a class which was born and died very quickly.  
A point which is far from the diagonal had a longer life.  
Depending on the context, this may mean the class  is more important, 
or more telling of the inherent topology of the space.  
See Fig.~\ref{F:circleExample} for an example.

\subsection{The Space $(\DD_p,W_p)$}

In order to define a framework for statistics, we will ignore the connection to topological spaces or maps between vector spaces and instead focus on the space of persistence diagrams abstractly. 
\begin{defn}
An abstract persistence diagram is a countable multiset of points along with the diagonal, $\Delta = \{(x,x) \in \R^2 \mid x \in \R\}$, with points in $\Delta$ having  countably infinite multiplicity.
\end{defn}
The distance between these abstract diagrams is the $p^{\textrm{th}}$ Wasserstein distance.

\begin{defn}\label{D: Wass}
The $p^\textrm{th}$ Wasserstein distance between two persistence diagrams $X$ and $Y$ is given by 
\begin{equation*}
 W_p[\sigma](X,Y) := \inf_{\phi:X \to Y} \left[  \sum_{x \in X} \sigma(x,\phi(x))  ^p \right]^{1/p}
\end{equation*}
where $1 \leq p \leq \infty$, $\sigma$ is a metric on the plane, and $\phi$ ranges over 
bijections between $X$ and $Y$. 
\end{defn}
 We often use $\sigma = L_q$.
 Notice that for $p = \infty$,
\begin{equation*}
 W_\infty[L_q](X,Y) := \inf_{\phi:X \to Y} \sup_{x \in X} \norm{x-\phi(x)}_q.
\end{equation*}
$W_\infty[L_\infty]$ is often referred to as the bottleneck distance.  
For the majority of this paper, we will be using $W_2[L_2]$, which we refer to as $W_2$ for brevity.
We also assume that $\|\cdot\|$ implies $L_2$ distance.



\begin{defn}
The space of persistence diagrams $\DD_p$  consists of abstract persistence diagrams with finite distance to the empty diagram $D_\emptyset$, which is the diagram which consists of only the points on the diagonal.  
That is,
\begin{equation*}
 \DD_p = \{X \mid W_p (X, D_\emptyset ) < \infty\}
\end{equation*}
along with the $p^\textrm{th}$-Wasserstein metric, $W_p = W_p[\sigma]$, from Definition 
\ref{D: Wass}.
\end{defn}

The authors in \cite{Mileyko2011} show that $(\DD_p, W_p[L_\infty])$ is a Polish (complete and separable) space.
They also give a description of all of the compact sets in this space.
In \cite{Turner2011}, it is shown that these results hold if we work instead with $(\DD_2,W_2[L_2]) = (\DD_2,W_2)$, with the added benefit that it is a non-negatively curved Alexandrov space (a geodesic space with a lower bound on curvature).
Thus, every pair of diagrams has a minimal geodesic between them and this geodesic can be defined using a matching between the diagrams which minimizes Wasserstein distance.
So, for the remainder of the paper, we will focus on the space $(\DD_2,W_2)$.

%

\subsection{Vineyards}
The first definitions of vineyards \cite{Cohen-Steiner2006,Morozov2008} were used in the well-behaved case of a homotopy between two functions.  
In this case,  each off-diagonal point of a diagram varies continuously in time
and is called a vine.  
Vines can start and end at off diagonal points at times $0$ or $1$, or have starting or ending points on the diagonal for any $t$, 
see Fig.~\ref{F: Vineyard}.
 
As we do with persistence diagrams, let us  consider the space of abstract vineyards to be
the space of paths in persistence diagram space.
\begin{defn}
 The space of abstract vineyards is 
 \begin{equation*}
  \VV_2 = \{v:[0,1] \to \DD_2\mid v\textrm{ is continuous}\},
 \end{equation*}
 the space of continuous maps from the unit interval to $\DD_2$ where $v$ is continuous with respect to $W_2$.
\end{defn}


\begin{figure}
 \centering
 \includegraphics[width = .3\textwidth]{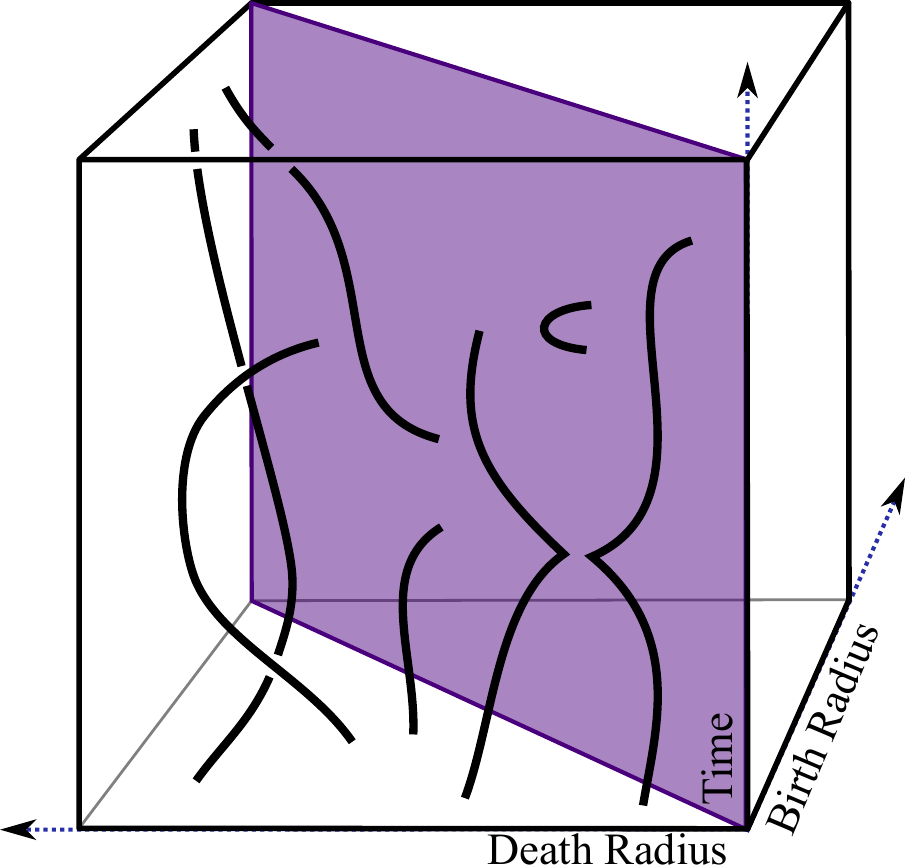}
\caption[A vineyard]{An example of a vineyard. For each time, given on the $z$-axis, there is a persistence diagram.   Since vineyards arising from continuous point clouds are continuous, each point in the diagram traces out a path called a vine.  These vines can have endpoints on the starting or ending times, or on the plane which projects to the diagonal. }
 \label{F: Vineyard}
\end{figure}

\section{\Frechet Means of Diagrams}
\label{sec:FMD}

This section reviews previous definitions of the mean of a set  of diagrams \cite{Mileyko2011} and an algorithm to compute the mean \cite{Turner2011}.
We will define the mean of a diagram as the Fr\'echet mean, give the algorithm for the computation of this mean, and finally present the non-uniqueness problem.

\subsection{\Frechet Means}

The \Frechet mean generalizes the mean of a set of points or a distribution in Euclidean space to any metric space.
It can be thought of as a generalization of the arithmetic mean in that it minimizes the sum of the square distances to points in the distribution.
Given a probability space $(\DD_2, \BB (\DD_2), \PP )$ where $\BB(\DD_2)$ consists of Borel sets of $\DD_2$, we can define the \Frechet mean as follows.


\begin{defn} \label{D: Frechet}
 Given a probability space $(\DD_2, \BB (\DD_2), \PP )$,
 \begin{equation*}
  \begin{array}{rccc}
 F_\PP :& \DD_2 &\longrightarrow &  \R\\
  & X& \longmapsto & \int_{\DD_2} W_2 (X,Y)^2 \, d\PP (Y) 
  \end{array}
 \end{equation*}
is the \Frechet function.
 The quantity
 \begin{equation*}
\Var_\PP = \inf_{X \in \DD_2} 
\left[ F_\PP(X) 
\right]
\end{equation*}
is the \Frechet variance of $\PP$ and the set at which the value is obtained
 \begin{equation*}
\E(\PP) = \{X  \mid  F_\PP (X) = \Var_{\PP} \}
\end{equation*}
is the \Frechet expectation, also called \Frechet mean.
\end{defn}

The mean in this case need not be a single diagram, but may be a set of diagrams.
In fact, there is no guarantee that $\E(\PP)$ is even non-empty.  
However, it was proved in \cite{Mileyko2011} that the \Frechet mean for $(\DD_p,W_p[L_\infty])$ is non-empty for certain types of well-behaved probability measures on $\DD_p$, and this result can be immediately extended to $(\DD_2,W_2[L_2])$ as in \cite{Turner2011}.
\begin{thm}\label{Thm: NonEmpty Frechet 1}
 Let $\PP$ be a probability measure on $(\DD_2, \BB (\DD_2))$ with a finite second moment.
If $\PP$ has compact support, then $\E(\PP) \neq \emptyset$. 
\end{thm}
A similar result holds when the tail probabilities of the distribution $\PP$ decay fast enough,
see \cite{Mileyko2011} for details. 



\subsection{Matchings, Selections, and Groupings}

\begin{figure}
\centering
 \begin{subfigure}[b]{.4\textwidth}
  \includegraphics[width = \textwidth]{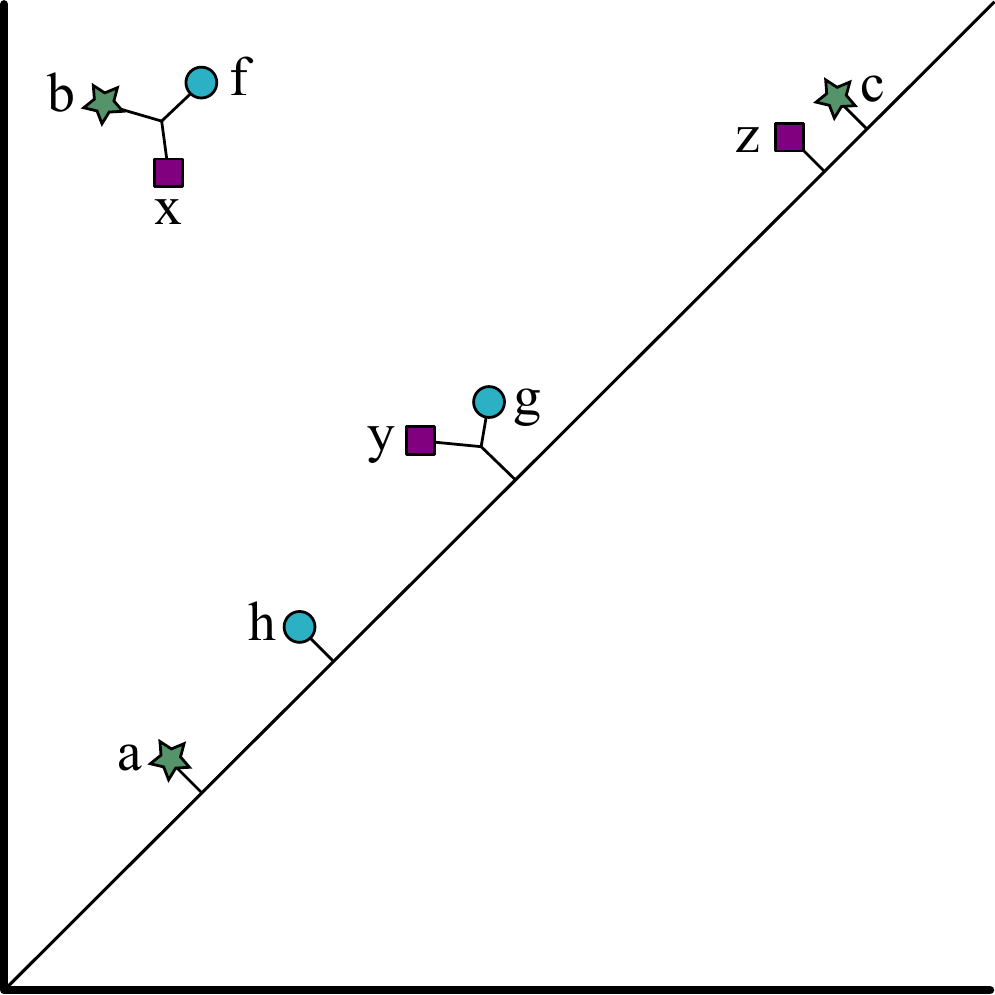}
  \caption{}
  \label{F:GroupingA}
 \end{subfigure}
 \qquad
 \qquad
 \begin{subfigure}[b]{.4\textwidth}
  \includegraphics[width = \textwidth]{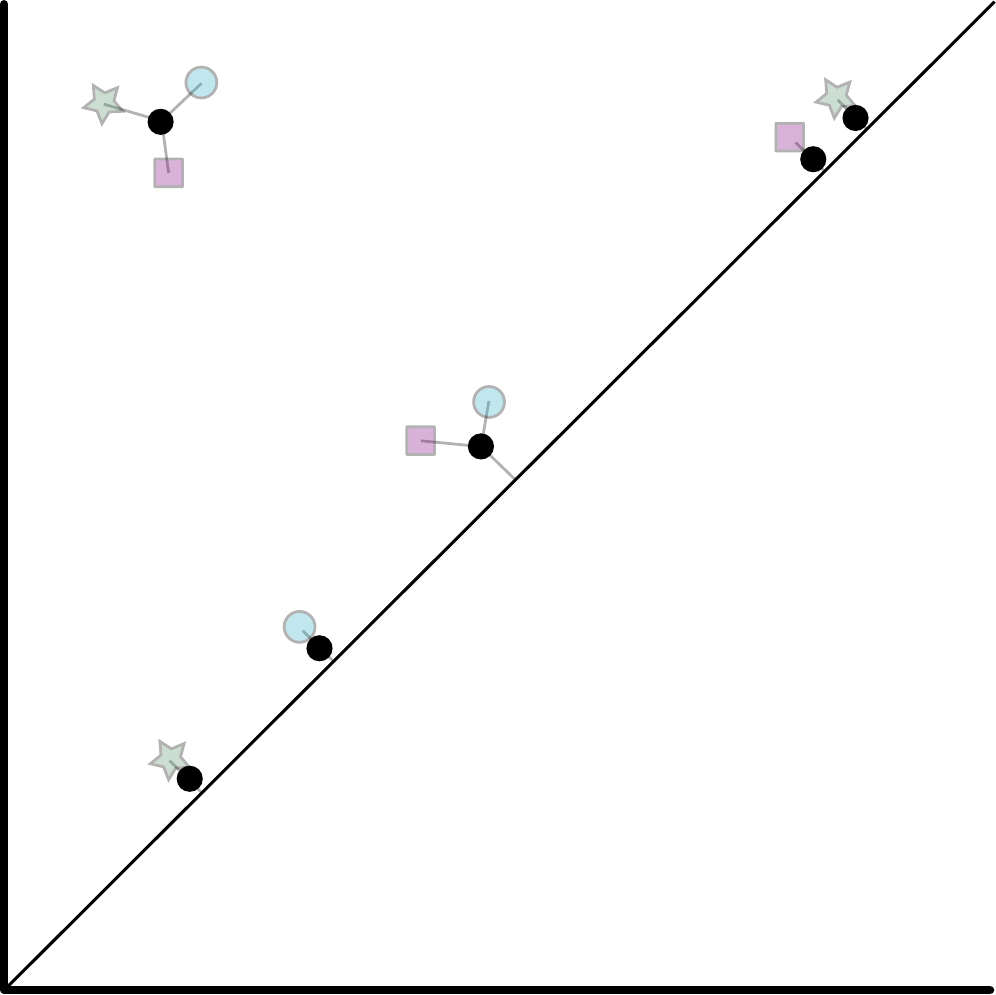}
  \caption{}
  \label{F:GroupingB}
 \end{subfigure}
 \caption[An example of a grouping for three persistence diagrams and the corresponding mean diagram.]{An example of a grouping for three overlaid persistence diagrams, $D_\blacksquare$, $D_\bigstar$, and $D_\bullet$ is given in (A). In this example, the grouping has four selections and the corresponding grouping matrix is given in Eqn.~\ref{E:GroupingEx}. The dark circles in diagram (B) give the mean diagram associated to this particular grouping.}
\label{F:Grouping}
\end{figure}

The focus of \cite{Mileyko2011} was to develop the probability theory required for statistical procedures on persistence diagrams, including defining a mean.
In \cite{Turner2011} an algorithm to compute an estimate of the \Frechet mean of a set of diagrams was given.
This algorithm centered around understanding an analogue to the Wasserstein distance matching in order to work with more than two diagrams.

The representation of a diagram for the purposes of these definitions is a list of its off-diagonal points, $X = [x_1,\cdots,x_k]$.  
We implicitly assume that every diagram has infinitely many copies of the diagonal.
Since our main theorem is stated with regards to diagrams with finitely many off-diagonal points, we will also implicitly assume that this list is finite.

\begin{defn}
Let $X = [x_1,\cdots,x_k]$ and $Y = [y_1,\cdots,y_m]$ be diagrams.
 A matching between $X$ and $Y$ is a bijection $\phi:X \to Y$.
 An optimal matching is one which attains the Wasserstein distance in Def.~\ref{D: Wass}.
\end{defn}

We now need to understand how to define matchings when we have $N$ diagrams instead of just two.  
For this, we define selections and groupings which restrict to matchings when $N=2$.

\begin{defn}
 Given a set of diagrams $X_1,\cdots,X_N$, a \textit{selection} is a choice of one point from each diagram, where that point could be $\Delta$.  
 The \textit{trivial selection} for a particular off-diagonal point $x \in X_i$ is the selection $m_x$ which chooses $x$ for $X_i$ and $\Delta$ for every other diagram.
 
 A \textit{grouping} is a set of selections so that every off-diagonal point of every diagram is part of exactly one selection.
\end{defn}

A grouping for $N$ diagrams which has $k$ selections can be stored as a $k\times N$ matrix $G$ where entry $G[j,i] = x$ means that the $j^\textrm{th}$ selection has point $x \in X_i$. 
See Fig.~\ref{F:Grouping} for an example;   
in this case, the grouping shown is given by the matrix
\begin{equation}\label{E:GroupingEx}
\bordermatrix{
   & D_\bigstar & D_\blacksquare  &D_\bullet \cr
1  & b &  x  & f \cr
2  & a  &  \Delta & \Delta \cr
3  & \Delta & y & g \cr
4  & \Delta  &   z       &\Delta \cr
5  & \Delta & \Delta & h \cr
6  & c & \Delta & \Delta
                } 
\end{equation}
where $\Delta$ represents the diagonal.
Note that we consider grouping to be equivalent up to reordering of the selections, and we can add or remove as many $(\Delta,\Delta,\cdots,\Delta)$ rows as we want.

The \textit{mean of a selection $s$} is the point denoted $\mean(s)$ which minimizes the sum of the square distances to the elements of the selection. 
When necessary, the notation $\mean_X(s)$ is used to emphasize the diagram set of interest.
The computation of this point will be discussed in Sect.~\ref{S:MeansSelections}.

The \textit{mean of a grouping}, $\textrm{mean}(G)$, is a diagram in $\DD_2$ with a point at the mean of each selection.
When it is unclear as to the set of diagrams from which this mean arose, we will denote it as $\textrm{mean}_X(G)$.
Note that the mean of the selection yields a point while the mean of a grouping yields a diagram.

It should be noted that there are close ties between these groupings and the \Frechet mean. 
The diagrams in the \Frechet mean (and all other local minima of the \Frechet function) are $\mean(G)$ for some grouping $G$ \cite{Turner2011}.
Thus, we define an \textit{optimal grouping} $G$ to be one such that $\mean_X(G)$ is in the \Frechet mean.

\subsection{Issues with extensions to Vineyards}
\label{sec:Problem}

The algorithm given in \cite{Turner2011} utilizes the tight relationship between groupings and their means in order to find a local minimum of the \Frechet function.
Assume we are working with a set of diagrams $X_1,\cdots,X_N$.
The idea of the algorithm is to find a candidate diagram $Y$ for the mean, compute a minimal matching for each pair $(Y,X_i)$, and build up these matchings into a grouping for the whole set $X_1,\cdots,X_N$.
Then the candidate diagram $Y$ is replaced with the mean of the computed grouping and the process is repeated until it terminates. 
See Appendix \ref{Appendix:Algorithms} for details.

However, it is important to note that the non-uniqueness of optimal groupings leads to non-uniqueness of the \Frechet mean.
Consider the example of two diagrams in Fig.~\ref{F:SquareExampleVineyard}A.
Their mean diagrams are given in  \ref{F:SquareExampleVineyard}B.
In Fig.~\ref{F:SquareExampleVineyard}A, there are two persistence diagrams overlaid: $D_\blacksquare$ has square points $1$ and $2$, $D_\bullet$ has circle points $a$ and $b$.  
Since the four points lie exactly on a square, the grouping (which can also be called matching in this instance) to give the Wasserstein distance could either be $\{(a,1), (b,2)\}$ or $\{(a,2),(b,1)\}$.
Thus there are two diagrams which give a minimum of the \Frechet function: the diagram with $u$ and $v$, or the diagram with $x$ and $y$.

If  two vineyards pass through this configuration, the mean of the vineyards constructed by finding the mean at each time will not be continuous. 
Consider for example two vineyards of two points each who start in the grayed configuration of Fig.~\ref{F:SquareExampleVineyard}C and move along the dotted line to the darkened configuration.
At the bend of the dotted line, the points are at the corners of a square, so as in the example of Fig.~\ref{F:SquareExampleVineyard}B,  there are two possible choices for the mean.  
One is close to the means from the previous times, and one is close to the means from the following times.

This means that it is not beneficial to define the mean of a set of vineyards $\{V_i:[0,1] \to \DD_2\}$ by using the pointwise \Frechet mean on the set of diagrams $\{V_i(t)\}$ for a fixed $t$ as there is no notion of continuity.
Thus we must be more creative with our definition.

\begin{figure}
 \centering
 \begin{subfigure}[b]{.3\textwidth}
  \includegraphics[width = \textwidth]{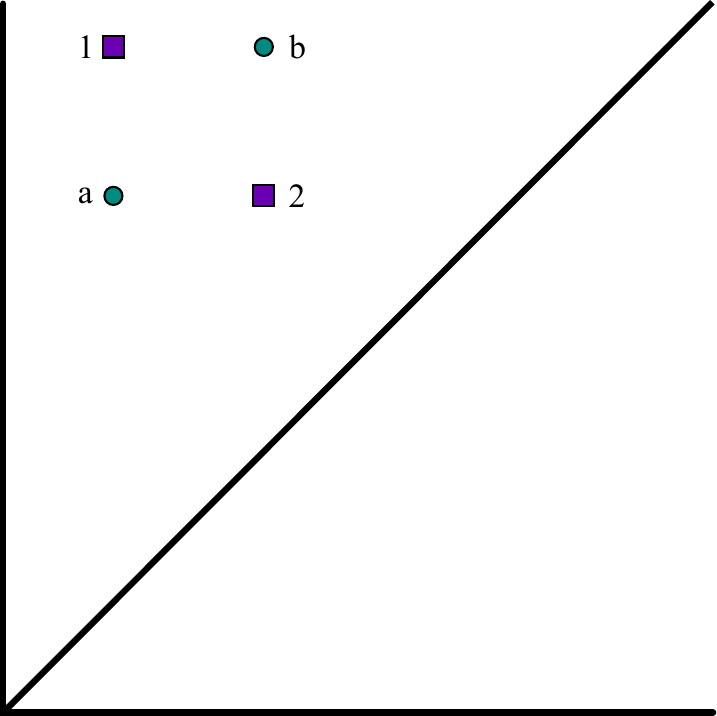}
  \caption{}
  \label{F:SquareExampleA}
 \end{subfigure}
 \hspace{16pt}
 \begin{subfigure}[b]{.3\textwidth}
  \includegraphics[width = \textwidth]{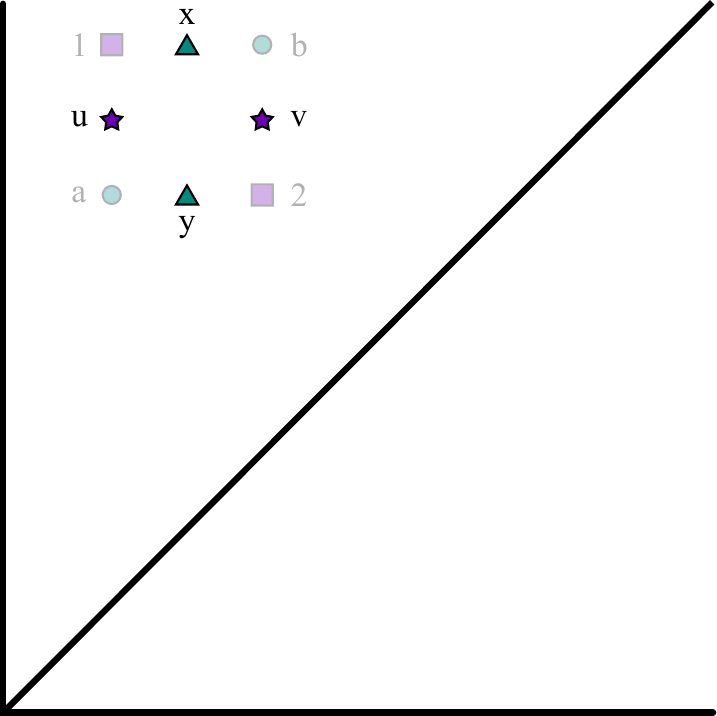}
  \caption{}
  \label{F:SquareExampleB}
 \end{subfigure}
%
 \hspace{16pt}
 \begin{subfigure}[b]{.3\textwidth}
  \includegraphics[width = \textwidth]{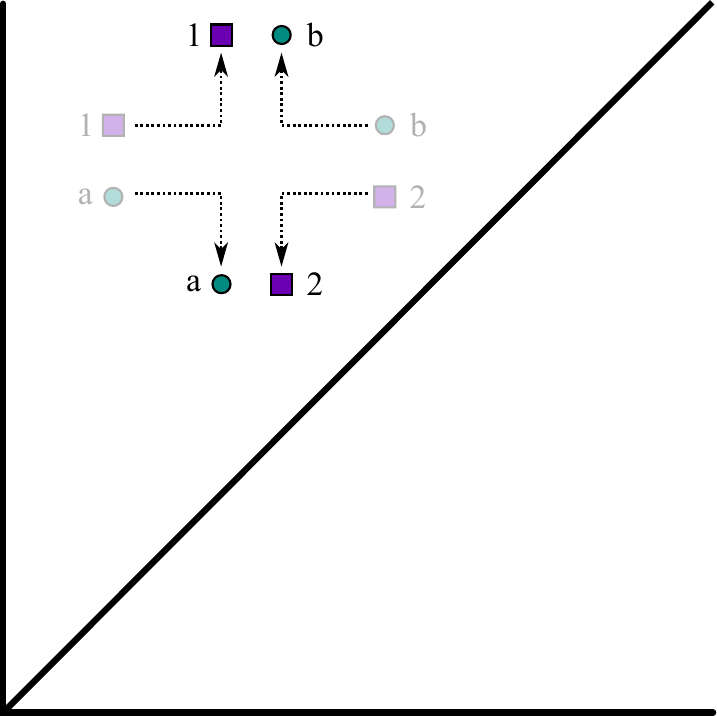}
  \caption{}
  \label{F:SquareExampleVineyardB}
 \end{subfigure}
 \caption[Two vineyards whose pointwise mean is not continuous.]{A counterexample to uniqueness of the \Frechet mean in $D_2$ gives an example of issue with means of vineyards.  Fig.~(A) shows two diagrams overlaid: $D_\blacksquare$ has points 1 and 2, $D_\bullet$ has points $a$ and $b$.  Since the grouping given by the Wasserstein distance is not unique, neither is the \Frechet mean.  The two possible means are given in Fig.~(B): one has points $x$ and $y$, the other has points $u$ and $v$. In Fig.~(C), we have two vineyards  which pass through the configuration of Fig.~(B).  The mean is continuous until the points get to the turn of the dotted line, where they form a square, and the mean jumps discontinuously.}
\label{F:SquareExampleVineyard}
\end{figure}

\section{The Mean as a Distribution}
\label{sec:MeanDist}

To overcome the lack of continuity of the mean vineyard and the non-uniqueness illustrated in Fig.~\ref{F:SquareExampleVineyard}, we will define the mean of a set of diagrams to be a distribution over persistence diagrams.
In order to prove continuity, the diagrams will be limited to $S_{M,K}$, the set of diagrams in $\DD_2$ with at most $K$ off-diagonal points, and all points $x =(x_1,x_2)$ satisfy $0 \leq x_1,x_2\leq M$.  

One important property of $S_{M,K}$ which we will utilize is that its diameter is finite.  
For a coarse bound, we note that for a diagram $X\in S_{M,K}$, any point is at most distance $\tfrac{\sqrt{2}}{2}M$ from the diagonal.  
Thus, $W_2(X,D_\emptyset) \leq \tfrac{\sqrt{2}}{2}KM$, and so the diameter is bounded by $\sqrt{2}KM$.
Obviously, however, this is a massive overestimate.

Consider the space $\PP(S_{M,K})$, the space of probability measures with finite second moment on $S \subset \DD_2$.  
This is of course a metric space with the standard probability Wasserstein distance as defined below.

\begin{defn} \label{D:WassProbability}
 The $p^\textrm{th}$-Wasserstein distance between two probability distributions, $\nu$ and $ \eta$, on metric space $(\X,d_\X)$ is
 \begin{equation*}
  \WW_p[d_\X](\nu,\eta) = \left[\inf_{\gamma \in \Gamma(\nu,\eta)} \int_{\X \times \X} d_\X(x,y)^p \, d\gamma(x,y)\right] ^{1/p}
 \end{equation*}
where $\Gamma(\nu,\eta)$ is the space of distributions on $\X \times \X$ with marginals $\nu$ and $\eta$ respectively. 
When $d_\X$ is obvious from context, we will instead write $\WW_p(\nu,\eta)$.
\end{defn}

Thus, we can use $\WW_2[W_2[L_2]]$ as the distance function on $\PP(S_{M,K})$, where the outside $\WW_2$ is the Wasserstein distance  of Def.~\ref{D:WassProbability} and the inside $W_2$ is the deterministic Wasserstein distance of Def.~ \ref{D: Wass}. 
Note that the map $Y \to \delta_Y$, where $\delta_Y$ is the delta measure concentrated on the diagram $Y$, gives an isometric embedding of $S_{M,K}$ into $\PP(S_{M,K})$.


This section is organized as follows.  
First, we give the intuition for the PFM in Section \ref{Sect:Intuition}, then we discuss the precise definition in Section \ref{Sect:H-example}, and finally work out an example in Section \ref{Sect:H-example}.

\subsection{Intuition}
\label{Sect:Intuition}

We first give an intuitive description of the main ideas behind the new definition. 
The basic ideas we use to achieve continuity of a mean diagram is to think of diagrams as probabilistic objects and groupings between diagrams as probabilistic objects. 
We track the probabilities of different  groupings being optimal for of a set of diagrams drawn by perturbing the points in the original input diagrams. 
The mean of a set of diagrams is not a diagram but a distribution over diagrams which are each the minimizer for some grouping. 
The weight on a diagram is the probability that the corresponding grouping is optimal.

The Fr\'echet mean is generically unique. 
More precisely, the measure of sets of diagrams in $(S_{M,K})^N$ with non-unique Fr\'echet means is of measure zero \cite{Turner2011}. 
However, the Fr\'echet mean is not stable. 
To see this, consider a slight perturbation of the point configuration in Fig.~\ref{F:SquareExampleVineyard}A.
The result is a mean which contains exactly one of $D_{\blacktriangle}$ and $D{_\bigstar}$ since the perturbation will result in exactly one groupings to be optimal. 

To address this problem, we consider diagrams as probabilistic objects.
So, given a diagram with labeled off diagonal points $\{p_1,...,p_\ell\}$ we consider the collection as a probability density function on sets of labeled points $\{x_1, x_2, \ldots x_\ell\} \subset \R^2\cup \Delta$ where each $x_i$ is a perturbation of $p_i$ and can be either off the diagonal or a copy of the diagonal.  This is the "trembling hands" of the points considered as players in game theory; they sometimes play a slightly different strategy than the one they intended to play. 

We will use the density function for an individual perturbation to be a linear combination of a uniform distribution over a small ball centered at $p_i$ and the Dirac function over the diagonal. This is for ease of calculations and there is no theoretical restriction.
%

For example, consider Fig.~\ref{F:SquareExampleVineyard}A. After slightly perturbing the points in the diagrams, the probability of the two optimal matches is about equal. This will result in a mean diagram of 
$$p_{\blacktriangle} \cdot \delta_{D_{\blacktriangle}}+ p_{\bigstar} \delta_{D{_\bigstar}},$$
where $p_{\blacktriangle} = 1- p_{\bigstar} \approx .5.$

In general, if $X = \{X_1, \cdots, X_N\}$ is a set of diagrams from $S_{M,K}$, we define its mean to be the following distribution on $S_{M,NK}$:
\begin{defn}
The probabilistic \Frechet mean (PFM) for a set of diagrams $\{X_1,\cdots,X_n\}$ is the distribution given by
\begin{equation*}
 \mu_X = \sum_G \P(\HH_X = G) \cdot \delta_{\textrm{mean}_X(G)}.
\end{equation*}
\end{defn}

Here the sum is taken over all possible groupings $G$ on the set of diagrams, and $\textrm{mean}_X(G)$ is the mean diagram for the specific grouping $G$.
The weights $\P(\HH = G)$ are derived from a random variable $\HH$ which can be thought of either as a probabilistic grouping where each point in the input diagram is replaced with a localized distribution centered on the point or as the  probability that a stochastic perturbation of the input  diagrams would lead to $G$ being the optimal grouping.
We now explain this in more detail.

\subsection{The Definition of $\HH$} \label{match}
\label{Sect:DefnH}

\begin{figure}
 \centering
 \includegraphics[width = .48\textwidth]{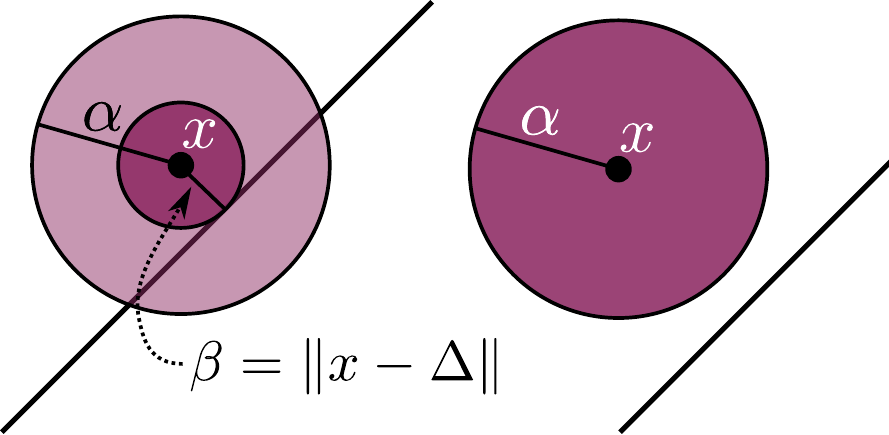}
 \caption[The method for drawing points.]{The method for drawing points.  
 For a point $x \in X_i$  where $\|x-\Delta\|\geq \alpha$ as at right, a point is drawn from the uniform distribution on the ball of radius $\alpha$ centered at $x$ 
 This point is then added to the diagram $X_i'$.  
 For a point $x \in X_i$ where $\|x-\Delta\|<\alpha$ as at left, a point is still drawn from ball of radius $\alpha$, however the point is only added to $X_i'$ if it is inside the ball of radius $\beta = \|x-\Delta\|$ centered at $x$.
 }
 \label{F: DrawNearDiagonal}
\end{figure}
We are given a set $X = \{X_1,\cdots,X_N\}$ of diagrams from $S_{M,K}$.  
We now define $\HH$, a grouping valued random variable.
As with other notation, we will write $\HH_X$ when it is necessary to distinguish the diagrams of interest.

First, choose parameter $\alpha>0$ which determines how much we perturb the points.
Label all the points within the $X_i$. When we perturb the points in the $X_i$ remember the labels. 
We need to define the distributions $\eta_{x}$ for each $x \in X_i$  which gives the method for perturbing the off-diagonal points of the persistence diagrams. 
This is based on the uniform distribution over the ball of radius $\alpha$, but must be modified if the initial point is within $\alpha$ of the diagonal.  
This can be thought of as instead adding a point on the diagonal if the drawn $x'$ is outside of $B(x,\|x-\Delta\|)$; however, since we only keep track of off-diagonal points for storing a persistence diagram,  this point can safely be ignored.
Note that if $x$ is more than $\alpha$ away from the diagonal, an off-diagonal point is always added to the diagram.  
However, the probability that an off-diagonal point gets added decreases as the distance to the diagonal decreases.
See Fig.~\ref{F: DrawNearDiagonal}.

Let $x \in X_i$ and set $r=\min\{\alpha, \|x-\Delta\|\}$. 
Define the density function on $\R^2\cup \Delta$ for $x$ as
\begin{align*}
\eta_{x}=\frac{1}{\pi \alpha^2}\mathbf{1}_{B(x,r)}+ \frac{\alpha^2-r^2}{\alpha^2}\delta_{\Delta}.
\end{align*}
If we wished, we could work with a more general distribution such as a normal distribution restricted to a disc $B(x,r)$; however, for clarity of the proof we will use the uniform distribution.


We now draw perturbations $X_1',\cdots,X_N'$ of original input diagrams $X_1, X_2, \cdots, X_N$ by drawing the $x'$ from the corresponding density functions $\eta_{x}$.
%
We are interested in the optimal groupings for these diagrams where we have kept track of the labeling. 
Since each point in a draw of $X_i'$ is associated to a point in $X_i$, we associate a grouping of the $\{X_i'\}$ with the grouping using the corresponding points of $X_i$.  
To make sure this is well-defined, whenever a copy of the diagonal is used in a selection within a grouping for $X'_1, X'_2, \ldots X'_N$ we assume that it came from an unlabeled copy of the diagonal  (not from an off-diagonal point in some $X_i$). 
Exactly the points in the $X_i$ which are sent to an off-diagonal point in $X'_i$ are contained in some selection.
Some points in the $X_i$ did not get corresponding off-diagonal points in the draw $X_i'$ and therefore are not represented in the selections. 
We can extend this set of selections to a full grouping by adding the trivial selection for these points. 
That is, if a point $x \in X_i$ did not lead to a off-diagonal point in $X_i'$, we add the selection which chooses $x$ for $X_i$ and $\Delta$ for every other diagram.

Note that the diagram $\mean_X(G)$, where $X\in (S_{M,K})^N$ can have at most $NK$ points and those points will be contained in a box of size $M$ (as their coordinates are an affine combination of values in $[0,M]$).
Thus, $\mu_X$ is an element of $\PP(S_{M,NK})$.

With probability one there is a unique optimal grouping for $X_1',\cdots,X_N'$ \cite{Turner2013}. 
This implies that we have constructed a grouping valued random element which we call $\HH$.

\subsection{Example} 
\label{Sect:H-example}

\begin{figure}
 \centering
 \begin{subfigure}[b]{.4\textwidth}
  \includegraphics[width = \textwidth]{PersistenceDiagram-Grouping}
  \caption{}
  \label{F:GroupingDraw1}
 \end{subfigure}
 \qquad
 \begin{subfigure}[b]{.4\textwidth}
  \includegraphics[width = \textwidth]{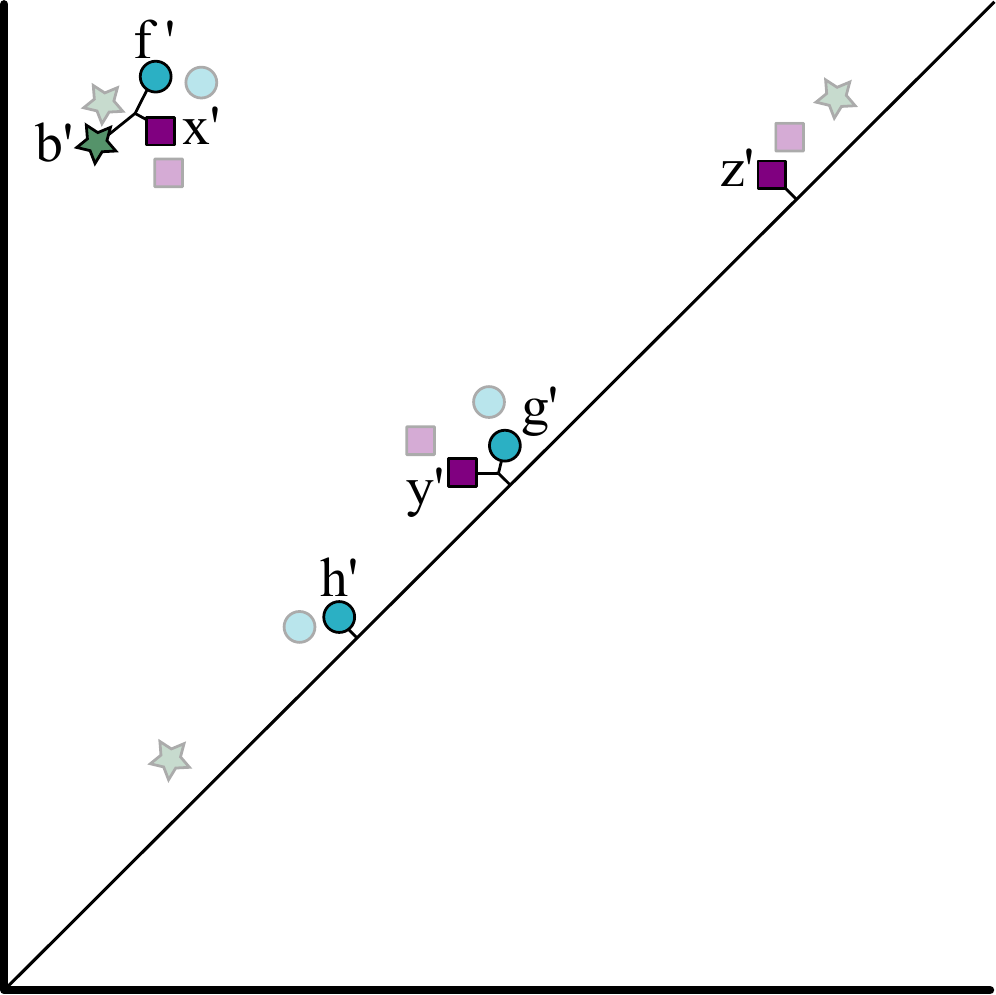}
  \caption{}
  \label{F:GroupingDraw2}
 \end{subfigure}
 \caption[An example of corresponding groupings for a given draw.]{An example of corresponding groupings for a given draw.  The original diagrams are $D_\bigstar$, $D_\Box$ and $D_\bigcirc$ in Fig.~(A). A point is drawn near each point away from the diagonal, and points are drawn for some points near the diagonal to construct $D_\bigstar'$, $D_\Box'$ and $D_\bigcirc'$ in Fig.~(B).  The grouping for the mean of these three diagrams is computed using Algorithm \ref{Alg:KatesMean} and the associated grouping is given in Eqn.~\ref{E:GroupingPerturb}. Then the grouping is converted to a grouping for $D_\bigstar$, $D_\Box$ and $D_\bigcirc$ in Eqn.~\ref{E:GroupingPerturb2} and drawn in Fig.~(A).}
\label{F:GroupingDraw}
\end{figure}

Here is an example to make the discussion above a little more clear. 
Consider the three overlaid diagrams in Fig.~\ref{F:GroupingDraw}A.  
Points are drawn in the ball of radius $\alpha$ centered at each point.
  Since $a,c, h, g,y,$ and $z$ are near the diagonal there is a chance that the diagonal is drawn for them.  
In this particular draw, given in Fig.~\ref{F:GroupingDraw}B, the diagonal is drawn for $a$ and $c$.  

For the diagrams in Fig.~\ref{F:GroupingDraw}B, the optimal grouping shown is 
\begin{equation}\label{E:GroupingPerturb}
\bordermatrix{
   & D_\bigstar' & D_\Box'  &D_\bigcirc' \cr
1  & b' &  x'  & f' \cr
2  & \Delta & y' & g' \cr
3  & \Delta & \Delta & h' \cr
4  & \Delta & z' & \Delta
                } .
\end{equation}
So, to find the corresponding grouping for the original diagrams, we replace each point with its corresponding point, and add in the trivial selection for the points that were not chosen:
\begin{equation}\label{E:GroupingPerturb2}
\bordermatrix{
   & D_\bigstar & D_\Box  &D_\bigcirc \cr
1  & b &  x  & f \cr
2  & \Delta & y & g \cr
3  & \Delta & \Delta & h \cr
4  & \Delta & z & \Delta \cr
5  & a & \Delta & \Delta \cr
6  & c & \Delta & \Delta
                } .
\end{equation}

%

\section{Continuity}
\label{sec:Theorems}

In this section, we prove our main theorem: that the mean distribution varies continuously when faced with a continuously varying set of input diagrams.
We first must give $(S_{M,K})^N$ and $\PP(S_{M,NK})$ metrics.
\begin{defn}
Let  $X = \{X_1,\cdots,X_N\}$ and $Y = \{Y_1,\cdots,Y_N\}$ be elements of $(S_{M,K})^N$.  
Then the space $(S_{M,K})^N$ is given the metric $\displaystyle{\overrightarrow{d_2}(X,Y) = \left(\sum_{i=1}^N W_2(X_i,Y_i)^2\right)^{1/2}}$. 
\end{defn}

\begin{defn}
The space $\PP(S_{M,NK})$ consists of  distributions with the Wasserstein metric $\WW_2$ of Def.~\ref{D:WassProbability}.  
\end{defn}

With these structures, we can state our main theorem.
\begin{thm}\label{Thm:Main}
Let $X = (X_1 ,\cdots , X_N )$ and $Y = (Y_1, \cdots  , Y_N )$ denote sets of diagrams in $(S_{M,K} )^N$ with PFMs $\mu_X$ and $\mu_Y$ respectively
and let $\phi_i:X_i \to Y_i$ for $i\in\{1,\cdots,N\}$ be any set of optimal matchings between the pairs.
Let $\tX_i$ be the diagram consisting of points $x \in X_i$ such that $\phi_i(x) = \Delta$. 
Likewise, let $\tY_i$ be the diagram consisting of points $y \in Y_i$ such that $\phi\inv(y) = \Delta$.
   
Then
\begin{align*}
 \WW_2(\mu_X , \mu_Y ) 
 & \leq  
 C\left(  \sum_{x \in \tX} \|x-\phi(x)\| + \overrightarrow d_2(X,Y)^2 \right)^{\tfrac{1}{2}}
\end{align*}
 where $C = \sqrt{\left( 2 \left(\tfrac{1}{N^2} + \tfrac{\overline{M}^2}{\alpha^2}\right)+  \tfrac{4\overline M^2}{\alpha}+ 1\right)} $.
\end{thm}

Before delving into the proof, we note that there are two major results that are now corollaries to this theorem.  
First, we have that the map which takes a set of diagrams to its PFM is \Holder continuous.
\begin{cor}
\label{Cor:MapCont}
 Let 
 \begin{center}
 \begin{tabular}{rccc}
  $\Phi:$& $(S_{M,K})^N$  &$\longrightarrow$&  $\PP(S_{M,NK})$\\
        & $(X_1,\cdots,X_N)$ &$\longmapsto$ & $\mu_{X}$
 \end{tabular}
\end{center}
be the map which sends a set of diagrams to its PFM.
Then $\Phi$ is \Holder continuous with exponent $1/2$
and constant $C' = \max\{ C\sqrt{NK+1}, \overline M \}$.
That is, 
\begin{equation*}
\WW_2(\mu_X,\mu_Y)\leq C'\sqrt{{\overrightarrow{d_2}(X,Y)}}
\end{equation*}
for all $X,Y \in (S_{M,K})^N$.

\end{cor}

\begin{proof}
 First, notice that for any $x \in \tX_i$, $\|x-\phi(x)\| \leq W_2(X_i,Y_i) \leq \overrightarrow{d_2}(X,Y)$ and the number of off-diagonal points in $\tX$ is at most $NK$.
 Thus, if $\overrightarrow{d_2}(X,Y) <1$, the theorem implies that 
 \begin{align*}
  \WW_2(\mu_X,\mu_Y) 
    & \leq C \left(NK\overrightarrow{d_2}(X,Y) + \overrightarrow{d_2}(X,Y)^2   \right)^{1/2}\\
    & \leq C(NK+1)^{1/2} \sqrt{\overrightarrow{d_2}(X,Y)}.
 \end{align*}
 If $\overrightarrow{d_2}(X,Y) \geq 1$, then we note that since $\WW_2(\mu_X,\mu_Y) \leq \overline M$,  $\WW_2(\mu_X,\mu_Y) \leq \overline M \sqrt{\overrightarrow{d_2}(X,Y)}$ trivially.

\end{proof}

Secondly, if we define a mean pointwise for a set of vineyards, we can stitch the results together into a continuous path in $\PP(S_{M,NK})$.
This result is immediate given the previous corollary.
\begin{cor}
\label{Cor:VineCont}
 Let $\gamma_1,\cdots,\gamma_N:[0,1]\to S_{M,K}$ be vineyards in $\VV_2$ restricted to $S_{M,K}$.
 Then 
 \begin{equation*}
  \begin{array}{rccc}
   \mu_\gamma:&[0,1] &\longrightarrow& \PP(S_{M,NK})\\
   & t & \longmapsto & \mu_{\gamma_1(t),\cdots,\gamma_N(t)}   
  \end{array}
 \end{equation*}
 is continuous.
\end{cor}

%
%

\section{Proof of Theorem~\ref{Thm:Main}}

The proof of Thm.~\ref{Thm:Main} is  organized as follows.  
We start by discussing means of selections in Section \ref{S:MeansSelections}.
Then, in order to prove the theorem, we need to match groupings for $X$ to groupings for $Y$.  
 We do this in Section \ref{S:ProofA}  by first working with those off-diagonal points in $X$ which are close to off-diagonal points in $Y$ where there is an obvious association between groupings in $X$ and groupings in $Y$.  
Then, the issue of points matched to the diagonal is discussed in Section \ref{S:ProofB}, and finally, the full proof is given in Section \ref{S:ProofC}.

\subsection{Means of selections}
\label{S:MeansSelections}
Consider the mean of the selection $s$ consisting of $N$ points: $\{p_1,\cdots,p_k\}$ with $p_i = (x_i,y_i)$ off-diagonal, and $N-k$ copies of the diagonal $\Delta$.
A quick computation gives this point as 
\begin{align} \label{E: Mean of a match}
 \textrm{mean}_X(s) = \frac{1}{2Nk}\Big(  &(N+k) \sum_i x_i + (N-k) \sum_i y_i,\\
	      &(N-k) \sum_i x_i + (N+k) \sum_i y_i\Big). \notag    
\end{align}

Sometimes it may be simpler to consider the mean of two selections  in rotated coordinates with axes $(1/\sqrt{2},1/\sqrt{2})$ and $(-1/\sqrt{2},1/\sqrt{2})$.
Writing $p_i = (a_i,b_i)$ in these coordinates, Eqn.~\ref{E: Mean of a match} becomes  
\begin{align*}
 \textrm{mean}_X(s) &= \left( \frac{1}{k} \sum_{i=1}^k a_i,\,\,\, \frac{1}{N} \sum_{i=1}^k b_i\right).
\end{align*}
In these coordinates, it is easier to see what happens for the mean of a trivial selection.  
\begin{lemma}
\label{Lemma:MeanOfTrivial}
 For any off-diagonal point $x$, $\|\mean(s_x)-\Delta\| = \tfrac{1}{N}\|x-\Delta\|$.
\end{lemma}
\begin{proof}
If the single off-diagonal point is at $x=(a,b)$ in the rotated coordinates and there are a total of $N$ diagrams, 
\begin{equation*}
 \textrm{mean}_X(s_x) = \left(a, \frac{1}{N}b \right)
\end{equation*}
so the distance to the diagonal is minimized at the point $(a,0)$, again in rotated coordinates.
Thus $\|x-\Delta\| = b$ and hence
\begin{equation*}
\|\mean(s_x)-\Delta\|^2 = (\tfrac{1}{N}b)^2 = \tfrac{1}{N^2} b^2 = \tfrac{1}{N^2}\|x-\Delta\|^2.
\end{equation*}

\end{proof}

To conclude the section, we show that the distance between points which are the means of selections is bounded by the distance between the points which build the selections.
\begin{lemma}\label{Lemma:CloseMatchings}
 Let $s$ be a selection of points $z_1, z_2,\cdots,z_N$ and $\hat s$ a selection of points $\hat z_1,\hat z_2,\cdots,\hat z_N$.
 Further, assume $z_1,\cdots,z_k$ and $\hat z_1, \cdots \hat z_k$ are all off-diagonal for $1 \leq k \leq N$ and 
 $z_{k+1} = \cdots = z_{N} = \hat z_{k+1} = \cdots =\hat z_N = \Delta$.
 Then
 \begin{equation*}
  \|\textrm{mean}(s)-\textrm{mean}(\hat s)\|^2 \leq \sum_{i=1}^{k}\|z_i-\hat z_i\|^2.
 \end{equation*}
\end{lemma}

\begin{proof}
Consider the means of the selections $s$ and $\hat{s}$ in rotated coordinates with axes $(1/\sqrt{2},1/\sqrt{2})$ and $(-1/\sqrt{2},1/\sqrt{2})$.
Writing $z_{i} = (u_i,v_i)$ and  $\hat z_{i} = (\hat u_i,\hat v_i)$ in these coordinates, we have  
\begin{align*}
 \textrm{mean}_{Z}(s) &= \left( \frac{1}{k} \left(\sum_{i=1}^k u_i\right), \frac{1}{N} \left(\sum_{i=1}^k v_i\right)\right),\\
 \textrm{mean}_{\hat{Z}}(\hat{s}) &= \left( \frac{1}{k} \left(\sum_{i=1}^k \hat u_i\right), \frac{1}{N} \left(\sum_{i=1}^k \hat v_i\right)\right),
\end{align*}
hence, using Cauchy-Schwartz,
\begin{align*}
\|\mean_{Z}(s)-\mean_{\hat{Z}}(\hat s)\|^2 
& =\left( \frac{1}{k}\sum_{i=1}^k (u_i-\hat u_i) \right)^2 + \left( \frac{1}{N}\sum_{i=1}^k (v_i-\hat v_i) \right)^2\\
& \leq \frac{1}{k} \sum_{i=1}^k (u_i-\hat u_i) ^2 +  \frac{1}{N}\sum_{i=1}^k (v_i-\hat v_i)^2\\
& \leq  \sum_{i=1}^k (u_i-\hat u_i) ^2 + \sum_{i=1}^k (v_i-\hat v_i)^2\\
& = \sum_{i=1}^k \|z_i-\hat z_i\|^2.
\end{align*}
\end{proof}


\subsection{Proof for off-diagonal points}
\label{S:ProofA}

In this section, our goal is Prop.~\ref{prop:meandistbound} where we bound the distance between the PFM of diagrams where we do not have to worry about the issue of points close to the diagonal.  
 Let $X=\{X_1, X_2, \ldots, X_N\}$ and $Y =\{Y_1, Y_2, Y_3, \ldots Y_N\}$ be in $(S_{M,K})^N$ where $\GG(X)$ and $\GG(Y)$ are the sets of possible groupings for $X$ and $Y$ respectively. 
For the sake of notation, let $\overline X = \bigcup_i X_i$ and  $\overline Y = \bigcup_i Y_i$ be the set of all off-diagonal points in the input sets

Let $\phi_i:X_i \to Y_i$ be optimal matching and consider the case where $X_i$ and $Y_i$ have the same number of off-diagonal points and that $\phi_i$ maps off-diagonal points to off-diagonal points. 
By remembering the labeling and using the fact that $\phi$ matches off-diagonal points to off-diagonal points, $\phi$ induces a bijection  $\phi: \GG(X) \to \GG(Y)$ which we can use to associate the probability masses in order to construct a transportation plan for Prop.~\ref{prop:meandistbound}.


First, we bound the difference in probability for associated groupings under $\phi$.
\begin{lemma}\label{lem:massbound}
Let $X=\{X_1, X_2, \ldots, X_N\}$ and $Y =\{Y_1, Y_2, Y_3, \ldots Y_N\}$ be in $(S_{M,K})^N$. 
Let $\phi_i:X_i \to Y_i$ be optimal matchings, which can be thought of as $\phi:\overline X \to \overline Y$. 
Let us assume that $X_i$ and $Y_i$ have the same number of off-diagonal points and that $\phi_i$ maps off-diagonal points to off-diagonal points. 
Then
\begin{equation*}
\sum_{G \in \GG(X)} \max\big\{ \P(\HH_X=G) - \P(\HH_Y=\phi(G)),\,\, 0 \big\}
\leq \frac{{4}}{\alpha}\sum_{x\in \overline X} \|x-\phi(x)\|.
\end{equation*}
\end{lemma}

\begin{proof}

The proof is immediate whenever $ \|x-\phi(x)\|\geq \alpha$ for some $x\in X$ as 
\begin{equation*}
\sum_{G \in \GG(X)} \max\big \{\P(\HH_X=G) - \P(\HH_Y=\phi(G)), \,\,0\big\}
\leq \sum_{G \in \GG(X)} \P(\HH_X=G) \leq 1. 
\end{equation*}
Thus, assume that $ \|x-\phi(x)\|<\alpha$ for all $x\in X$.

For inputs $X= \{X_1, X_2, \cdots X_N\}$ and $Y=\{Y_1, Y_2 \ldots Y_N\}$, we have probability density functions $f_X$ and $f_Y$ for the population of $N$ diagrams drawn by suitably perturbing the points within them. 
The random grouping valued element is determined by the populations of diagrams drawn by $f_X$ and $f_Y$. 
Thus, to bound $\displaystyle{\sum_{G \in \GG(X)} \max\{\P(\HH_X=G) - \P(\HH_Y=\phi(G)), 0\}}$ it is enough to bound 
\begin{equation*}
\int_{S_{M+\alpha,NK}^N} (f_X-f_Y)_+\, d\rho.
\end{equation*}

Let $m$ be the total number of points in the diagrams $X_i$, so $m = \#\{\overline X\}$. 
Construct a sequence of populations of diagrams $X=Z^0, Z^1, \ldots, Z^m=Y$ where at each stage, a point $x$ is moved to $\phi(x)$ while fixing the rest of the points. 
Let $f_k$ denote the probability density function over $S_{M+\alpha,NK}^N$ corresponding to $Z^k$.
Since
\begin{align}\label{eq:densitybound1}
\int_{S_{M,K}^N} (f_X-f_Y)_+\, d\rho &\leq \int_{S_{M,K}^N} \Big|f_X-f_Y\Big|\, d\rho
\leq \sum_{k=1}^m \int_{S_{M+\alpha,NK}^N} \Big|f_{k-1} -f_k\Big|\,d\rho,
\end{align}
we wish to bound $\int |f_{k-1} -f_k|\,d\rho$ in terms of $\|x-\phi(x)\|$ where $x$ is the point that was moved between $Z^{k-1}$ and $Z^k$. 

Let $r_x=\min\{\alpha, \|x-\Delta\|\}$.
As the perturbations of the points in the diagrams are independent we can integrate out the effects of the other points in the diagrams.
Therefore,
\begin{align*}\label{eq:bound1}
\int_{S_{M+\alpha,NK}^N} |f_{k-1} -f_k|\,d\rho&= \int_{\R^2 \cup \Delta} |\eta_{x}-\eta_{\phi(x)}| d\rho\\
&\leq \int_{\R^2} \left|\frac{1}{\pi \alpha^2}\mathbf{1}_{B(x,r_x)} - \frac{1}{\pi \alpha^2}
\mathbf{1}_{B(x_{\phi(x)},r_{\phi(x)})} \right| \, d\rho 
+ \left|\frac{\alpha^2-r_x^2}{\alpha^2}-\frac{\alpha^2-r_{\phi(x)}^2}{\alpha^2}\right| \\
&\leq \frac{1}{\pi \alpha^2} \int_{\R^2}  \mathbf{1}_{B(x,r_x)\triangle B(x_{\phi(x)},r_{\phi(x)})}\, d\rho +  \frac{|r_x^2-r_{\phi(x)}^2|}{\alpha^2}
\end{align*}
 where $U\triangle V$ denotes the symmetric difference of $U$ and $V$.

Now, $B(x,r_x)\cap B(x_{\phi(x)},r_{\phi(x)})$ contains a ball with diameter $r_x+r_{\phi(x)}-\|x-\phi(x)\|$.
Note also that $r_x, r_{\phi(x)}\leq \alpha$, and $|r_x-r_{\phi(x)}| \leq \|x-\phi(x)\|$.
Then 
\begin{align*}
\int_{\R^2}  \mathbf{1}_{B(x,r_x)\triangle B(x_{\phi(x)},r_{\phi(x)})}\, d\rho 
&\leq \pi r_x^2 + \pi r_{\phi(x)}^2 
  - 2\pi \Big(\tfrac{1}{2}(r_x+r_{\phi(x)}-\|x-\phi(x)\|)\Big)^2\\
&=\tfrac{\pi}{2}\big(
(r_x -r_{\phi(x)})^2 +\|x-\phi(x)\| ({2}r_x+{2}r_{\phi(x)}-\|x-\phi(x)\|) 
\big)\\
&\leq \tfrac{\pi}{2} \big(
\|x-\phi(x)\|^2 + \|x-\phi(x)\| (4\alpha -\|x-\phi(x)\|
)\big)\\
&\leq 2\pi\alpha\|x-\phi(x)\| 
\end{align*}
and therefore
\begin{align*}
\int_{S_{M+\alpha,NK}^N} |f_{k-1} -f_k|\,d\rho
&\leq \frac{\pi \alpha\|x-\phi(x)\|}{\pi \alpha^2} +  \frac{|r_x^2-r_{\phi(x)}^2|}{\alpha^2}\\
&\leq \frac{2\|x-\phi(x)\|}{\alpha} + \frac{ 2\alpha\|x-\phi(x)\|}{\alpha^2}\\
&\leq \frac{4\|x-\phi(x)\|}{\alpha}.
\end{align*}

Together with Eqn.~\eqref{eq:densitybound1} we can conclude that 
\begin{equation*}
\int_{S_{M+\alpha,NK}^N} |f_{X} -f_Y|\,d\rho \leq \sum_{x\in X} \frac{{4}\|x-\phi(x)\|}{\alpha}.
\end{equation*}

\end{proof}


Next, we use the previous lemma to bound the distance between the mean diagrams for associated groupings.

\begin{lemma}\label{lem:massbound1}
Let $X=\{X_1, X_2, \ldots, X_N\}$ and $Y =\{Y_1, Y_2, Y_3, \ldots Y_N\}$ be in $(S_{M,K})^N$. 
Let $\phi_i:X_i \to Y_i$ be an optimal matching.
Assume that $X_i$ and $Y_i$ have the same number of off-diagonal points and that $\phi_i$ maps off-diagonal points to off-diagonal points. 
Then 
\begin{equation*}
W_2(\mean_X G,\mean_Y \phi(G)) \leq \overrightarrow{d_2}(X,Y) 
\end{equation*}
for all $G\in \GG$.

\end{lemma}


\begin{proof}
Let $m_1,\cdots,m_\ell$ be the selections in $G$.
Thus $\phi(m_1),\cdots,\phi(m_\ell)$ are the selections of $\phi(G)$, $\mean(m_1),\cdots,\mean(m_\ell)$ are the off-diagonal points of $\mean_X(G)$ and  $\mean(\phi(m_1)),\cdots,\mean(\phi(m_\ell))$ are the off-diagonal points of $\mean_Y(\phi(G))$.
 Define a bijection $\psi:\mean_X G\to\mean_Y \phi(G)$ by sending $\mean(m_i)$ to $\mean(\phi(m_i))$.
 Thus, 
 \begin{align*}
W_2(\mean_X G,\mean_Y \phi(G))^2
&= \sum_{i=1}^{\ell} \|\mean(m_i) - \mean(\phi(m_i))\|^2\\
& \leq \sum_{i=1}^\ell \sum_{x \in m_i} \|x-\phi(x)\|^2\\
& = \sum_{i=1}^N \sum_{x \in X_i} \|x - \phi(x)\|^2\\
& = \sum_{i=1}^N W_2(X_i,Y_i)^2\\
& = \overrightarrow{d_2}(X,Y)^2.
 \end{align*}

\end{proof}

Finally, we can bound the distance between the PFMs for the sets of diagrams.
\begin{prop}\label{prop:meandistbound}
Let $X=\{X_1, X_2, \ldots, X_N\}, Y =\{Y_1, Y_2, Y_3, \ldots Y_N\} \in (S_{M,K})^N$ with PFMs $\mu_X, \mu_Y \in \PP(S_{M,NK})$ respectively. 
Let $\phi_i:X_i \to Y_i$ be optimal  matchings. 
Further, assume that $X_i$ and $Y_i$ have the same number of off-diagonal points and that $\phi_i$ maps off-diagonal points to off-diagonal points. 
Then
\begin{equation*}
\WW_2(\mu_X, \mu_Y) \leq  \left( \frac{4\overline{M}^2}{\alpha}\sum_{x\in X} \|x-\phi(x)\|\right) ^{1/2} + \overrightarrow{d_2}(X,Y)
\end{equation*}
where $\overline{M}$ is the maximal distance between diagrams in $S_{M,NK}$.
\end{prop}

\begin{proof}

%

Let $\GG(X)$  be the set of groupings on the set of diagrams $X=\{X_1, X_2, \ldots, X_N\}$ and  $\GG(Y)$ the set of diagrams $Y = \{Y_1, Y_2, \ldots Y_N\}$. 
Recall that 
\begin{equation*}
\mu_X=\sum_{G\in \GG(X)} \P(\HH_X=G) \delta_{\mean_X G}. 
\end{equation*}
Using the bijection  $\phi: \GG(X) \to \GG(Y)$,  we can write 
\begin{equation*}
\mu_Y=\sum_{G\in \GG(X)} \P(\HH_Y=\phi(G)) \delta_{\mean_Y \phi(G)}.
\end{equation*}
Thus
\begin{equation}
\begin{aligned}\label{eq:triineq1}
\WW_2(\mu_X, \mu_Y) 
=& \WW_2\left( \sum_{G \in \GG(X)} \P(\HH_X=G) \delta_{\mean_X G}, \sum_{G\in \GG(X)} \P(\HH_Y=\phi(G)) \delta_{\mean_Y \phi(G)}\right)\\
\leq& \WW_2\left( \sum_{G \in \GG(X)} \P(\HH_X=G) \delta_{\mean_X G},  \sum_{G \in \GG(X)} \P(\HH_Y=\phi(G)) \delta_{\mean_X G}\right)\\
 &+ 
\WW_2\left(\sum_{G \in \GG(X)} \P(\HH_Y=\phi(G)) \delta_{\mean_X G},\sum_{G\in \GG(X)} \P(\HH_Y=\phi(G)) \delta_{\mean_Y \phi(G)}\right)\\
\end{aligned}
\end{equation}
by the triangle inequality.

We want to bound the first term, 
$ \WW_2\left( \sum_{G \in \GG(X)} \P(\HH_X=G) \delta_{\mean_X G},  \sum_{G \in \GG(X)} \P(\HH_Y=\phi(G)) \delta_{\mean_X G}\right) $, 
by constructing a transportation plan which keeps most of the mass stationary.
Consider the following plan:
 \begin{itemize}
 \item If $\P(\HH_X=G)\leq \P(\HH_Y=\phi(G))$ then keep all the mass at $\delta_{\mean_X G}$ at the same spot. 
 \item   If $\P(\HH_X=G)>\P(\HH_Y=\phi(G))$ then keep $\P(\HH_X=G)$ worth of mass at  $\delta_{\mean_X G}$ and redistribute the rest as needed. 
 \end{itemize}
The amount of mass that moves is then $\sum_{G\in \GG} \max\{\P(\HH_X=G) - \P(\HH_Y=\phi(G)), 0\}$. 
By Lemma \ref{lem:massbound}, 
\begin{equation*}
\sum_{G \in \GG(X)} \max\{\P(\HH_X=G) - \P(\HH_Y=\phi(G)), 0\} \leq \frac{4}{\alpha}\sum_{x\in X} \|x-\phi(x)\|.
\end{equation*}
The distance between diagrams is bounded by $\overline{M}$. 
Therefore,
\begin{align*}
 \WW_2\left( \sum_{G \in \GG(X)} \P(\HH_X=G) \delta_{\mean_X G},  \sum_{G \in \GG(X)} \P(\HH_Y=\phi(G)) \delta_{\mean_X G}\right) ^2
 \leq \frac{4\overline{M}^2}{\alpha}\sum_{x\in X} \|x-\phi(x)\|.
 \end{align*}
 
 Focusing on the second term of Eqn. \ref{eq:triineq1},
\begin{align*}
\WW_2\left(\sum_{G \in \GG(X)} \P(\HH_Y=\phi(G)) \delta_{\mean_X G},\sum_{G\in \GG(X)} \P(\HH_Y=\phi(G)) \delta_{\mean_Y \phi(G)}\right) 
&\leq \max_{G\in \GG} W_2\Big(\mean_X G,\mean_Y \phi(G)\Big) \\
&\leq \overrightarrow{d_2}(X,Y)
\end{align*}
  by Lemma \ref{lem:massbound1}.
Together, this implies
\begin{align*}
\WW_2(\mu_X, \mu_Y) 
&= \WW_2\left( \sum_{G \in \GG(X)} \P(\HH_X=G) \delta_{\mean_X G}, \sum_{G\in \GG(X)} \P(\HH_Y=\phi(G)) \delta_{\mean_Y \phi(G)}\right)\\
&\leq \left(\frac{4\overline{M}^2}{\alpha}\sum_{x\in X} \|x-\phi(x)\|\right)^{1/2} +  \overrightarrow{ d_2}(X,Y).
\end{align*}

\end{proof}

\subsection{Proof for points close to the diagonal}
\label{S:ProofB}

In Section \ref{S:ProofA}, we were able to use the fact that a set of optimal matchings $\phi_i:X_i \to Y_i$ which associate off-diagonal points together induces a bijection $\phi:\GG(X) \to \GG(Y)$.  
However, if we have a point $x \in X_i$ such that $\phi_i(x) = \Delta$, $\phi(G)$ is no longer injective as different selections can map to the same selection.
For example given two diagrams, the grouping $G_1 = \{(x,\Delta), (y,\Delta)\}$ and $G_2 = \{(x,y)\}$ have the same image.
Thus, in the following proposition, we bound the distance between PFMs for diagrams which only differ by points that are matched to the diagonal. 
\begin{prop}\label{prop:losetodiag}
Let $X = (X_1,\cdots,X_N)$ and  $\tX = (\tX_1,\cdots,\tX_N)$ denote sets of diagrams in $(S_{M,K})^N$ where the off-diagonal points in each $\tX_i$ is a subset of those in $X_i$.
Then
\begin{align*}
\WW_2(\mu_X,\mu_{\tX})^2 
&\leq \left(\frac{1}{N^2} + \frac{\overline{M}^2}{\alpha^2}\right)\sum_{x\in X\backslash \tX}\|x-\Delta\|^2
\end{align*}
where $\overline{M}$ is the maximum distance between any two diagrams in $S_{M,NK}$.
\end{prop}
\begin{proof}
If $\|x-\Delta\|\geq \alpha$ for some $x \in X \setminus \tX$ then $\sum_{x\in X\backslash \tX}\|x-\Delta\|^2>\alpha^2$ and the theorem automatically holds as $\WW_2(\mu_X,\mu_{\tX})^2 \leq \overline{M}^2$ by the definition of $\overline{M}$. From now on assume that $\|x-\Delta\|<\alpha$ for all $x \in X \setminus \tX$.

Let $\GG(\tX)$ and $\GG(X)$ be the sets of groupings for $\tX$ and $X$ respectively.
There is an injection $i_{\tX} : \GG({\tX}) \hookrightarrow \GG(X)$ which maps a grouping $G \in \GG({\widetilde{X}})$ and  to the grouping in $\GG(X)$ which has all the same selections as $G$ along with the trivial selection for each unused point $x \in X_i \setminus \tX_i$. 
In order to bound $\WW_2(\mu_X,\mu_{\tX})$, construct a transportation plan from $\mu_{\tX}$ to $\mu_X$ as follows:
 \begin{itemize}
 \item If $\P(\HH_{\tX}=G)\leq \P(\HH_X=i_{\tX}(G))$ then move all the mass at $\delta_{\mean_{\tX} G}$ to $\delta_{\mean_{X} i_{\tX}(G)}$
 \item   If $\P(\HH_{\tX}=G)>\P(\HH_X=i_{\tX}(G))$ then move $\P(\HH_{X}=i_{\tX}(G))$ worth of mass from $\delta_{\mean_{\tX} G}$ to $\delta_{\mean_{X} i_{\tX}(G)}$ and redistribute the rest as needed. 
 \end{itemize}

First note that for any $G \in \GG(\tX)$, the amount of mass moved from $\delta_{\mean_{\tX} G}$ to its corresponding $\delta_{\mean_{X} i_{\tX}(G)}$ is bounded from above by $\P(\HH_{\tX}=G)$. 
Secondly, the amount of mass not moved from $\delta_{\mean_{\tX} G}$ to its corresponding $\delta_{\mean_{X} i_{\tX}(G)}$, is 
\begin{equation*}
\sum_{G\in \GG_{\tX}} \max\Big\{ \P(\HH_{\tX}=G)- \P(\HH_X=i_{\tX}(G)),\,\, 0\Big\}.
\end{equation*}
Therefore using this transport plan,
\begin{equation}
\begin{aligned}\label{eq:triineq2}
\WW_2(\mu_X,\mu_{\tX})^2 \leq &\sum_{G\in \GG_{\tX}} \P(\HH_{\tX}=G) W_2(\mean_{\tX} G, \mean_{X} i_{\tX}(G))^2\\
& + \overline{M}^2\sum_{G\in \GG_{\tX}} \max\{ \P(\HH_{\tX}=G)- \P(\HH_X=i_{\tX}(G)), 0\}
\end{aligned}
\end{equation}
where $\overline{M}$ is the maximum distance between any two diagrams in $S_{M,NK}$.

In order to bound $W_2(\mean_{\tX}(G), \mean_X(i_{\tX}(G)))^2$, observe that every off diagonal point that appears in $\mean_{\tX}(G)$ also appears in $\mean_X(i_{\tX}(G))$ and that the additional points in $\mean_X(i_{\tX}(G))$ correspond to the trivial selections $m_x$ for all $x \in X \setminus \tX$.
Each of these additional points are at distance $\|x - \Delta\|/N$ to the diagonal by Lemma \ref{Lemma:MeanOfTrivial}.
Thus, using the matching sending each of these additional points to the diagonal, 
\begin{equation*}
 W_2( \mean_{\tX}(G),\mean_X(i_{\tX}(G)))^2 \leq \sum_{x \in X \setminus \tX} \frac{\|x - \Delta\|^2}{N^2}. 
\end{equation*}
for all $G\in\GG_{\tX}$. 
Since $\sum_{G\in \GG_{\tX}} \P(\HH_{\tX}=G)=1$, 
\begin{equation*}
\sum_{G\in \GG_{\tX}} \P(\HH_{\tX}=G) W_2(\mean_{\tX} G, \mean_{X} i_{\tX}(G))^2
  \leq \sum_{x \in X \setminus \tX} \frac{\|x - \Delta\|^2}{N^2}.
\end{equation*}

Now we can consider the second half of Eqn.~\ref{eq:triineq2}.
Let $E$ be the event that all the points $x \in X\backslash \tX$ are perturbed to the diagonal. 
\begin{equation*}
\P(E)=\prod_{x\in X\backslash \tX} \left(1-\frac{\|x-\Delta\|^2}{\alpha^2}\right).
\end{equation*}
When $E$ is conditioned to be true, the randomly drawn diagrams for $X$ have the same distribution as that of $\tX$. This implies that $\P(\HH_X=i_{\tX}(G)|E)=\P(\HH_{\tX} = G).$

Thus, to bound $\sum_{G\in \GG_{\tX}} \max\{ \P(\HH_{\tX}=G)- \P(\HH_X=i_{\tX}(G)), 0\}$ observe that 
\begin{align*}
\P(\HH_X=i_{\tX}(G)) &> \P(\HH_X=i_{\tX}(G) \text{ and } E)\\
&=\P(\HH_X=i_{\tX}(G) | E)\P(E)\\
&=\P(\HH_{\tX}=G)\prod_{x\in X\backslash \tX} \left(1-\frac{\|x-\Delta\|^2}{\alpha^2}\right).
\end{align*}
This implies that 
\begin{align*}
\max\{ \P(\HH_{\tX}=G)- \P(\HH_X=i_{\tX}(G)), 0\} &\leq \P(\HH_{\tX}=G)\left(1-\prod_{x\in X\backslash \tX} \left(1-\frac{\|x-\Delta\|^2}{\alpha^2}\right)\right)\\
&\leq \P(\HH_{\tX}=G)\sum_{x\in X\backslash \tX}\frac{\|x-\Delta\|^2}{\alpha^2}
\end{align*} 
and hence
\begin{align*}
\sum_{G\in \GG_{\tX}} \max\{ \P(\HH_{\tX}=G)- \P(\HH_X=i_{\tX}(G)), 0\} &\leq \sum_{G\in \GG_{\tX}} \P(\HH_{\tX}=G)\left(\sum_{x\in X\backslash \tX}\frac{\|x-\Delta\|^2}{\alpha^2}\right)\\
&\leq\sum_{x\in X\backslash \tX}\frac{\|x-\Delta\|^2}{\alpha^2}.
\end{align*}

Substituting into \eqref{eq:triineq2} 
\begin{align*}
\WW_2(\mu_X,\mu_{\tX})^2 
&\leq 
\sum_{x \in X \setminus \tX} \frac{\|x - \Delta\|^2}{N^2} + \overline{M}^2\sum_{x\in X\backslash \tX}\frac{\|x-\Delta\|^2}{\alpha^2}\\
&\leq \left(\frac{1}{N^2} + \frac{\overline{M}^2}{\alpha^2}\right)\sum_{x\in X\backslash \tX}\|x-\Delta\|^2
\end{align*}
where $\overline{M}$ is the maximum distance between any two diagrams in $S_{M,NK}$.
 
\end{proof}

\subsection{Proof of Theorem \ref{Thm:Main}}
\label{S:ProofC}
With these results in hand, particularly Props.~\ref{prop:losetodiag} and \ref{prop:meandistbound}, we can prove the main theorem.

\begin{proof}[Proof of Thm.~\ref{Thm:Main}]
Let $X = (X_1 ,\cdots , X_N )$ and $Y = (Y_1, \cdots  , Y_N )$ denote sets of diagrams in $(S_{M,K} )^N$ with PFMs $\mu_X$ and $\mu_Y$ respectively.
We wish to find a constant $C$ such that $\WW_2 (\mu_X , \mu_Y ) \leq  C \overrightarrow{d_2}(X, Y )$.

For the moment assume that $\overrightarrow{d_2}(X, Y ) \leq 1$.
For each $i$, let $\phi_i:X_i \to Y_i$ be an optimal matching. 
Let $\tX_i$ be the diagram consisting of points $x \in X_i$ such that $\phi_i(x) = \Delta$. 
Likewise, let $\tY_i$ be the diagram consisting of points $y \in Y_i$ such that $\phi\inv(y) = \Delta$.
We will bound $\WW_2(\mu_X , \mu_Y )$ using the triangle inequality,
\begin{equation*}
 \WW_2(\mu_X , \mu_Y ) \leq  \WW_2(\mu_X , \mu_{\tX} ) +  \WW_2(\mu_{\tX} , \mu_{\tY} ) +  \WW_2(\mu_{\tY} , \mu_Y ).
\end{equation*}
Using Prop.~\ref{prop:losetodiag} for the first and third portions and Prop.~\ref{prop:meandistbound} for the second, we have
\begin{align*}
 \WW_2(\mu_X , \mu_Y ) 
 & \leq \left( \left(\frac{1}{N^2} + \frac{\overline{M}^2}{\alpha^2}\right)\sum_{x\in X\backslash \tX}\|x-\Delta\|^2\right)^{1/2}
+  \left(\left(\frac{4\overline M^2}{\alpha} \sum_{x \in \tX} \|x-\phi(x)\|\right)^{1/2} + \overrightarrow {d_2}(\tX,\tY) \right)  \\
 & \phantom{xxx}  + \left( \left(\frac{1}{N^2} + \frac{\overline{M}^2}{\alpha^2}\right)\sum_{y\in Y\backslash \tY}\|y-\Delta\|^2\right)^{1/2}.
\end{align*}
Let $U = \tfrac{1}{N^2} + \tfrac{\overline{M}^2}{\alpha^2}$ and $V = \tfrac{4\overline M^2}{\alpha}$.
Then via Cauchy-Schwartz,
\begin{align*}
 \WW_2&(\mu_X , \mu_Y )^2 \\
 &  \leq \left(U^{1/2}\left(\sum_{x\in X\backslash \tX}\|x-\Delta\|^2\right)^{1/2}
+  V^{1/2}\left(\sum_{x \in \tX} \|x-\phi(x)\|\right)^{1/2} + \overrightarrow {d_2}(\tX,\tY)   
  + U^{1/2}\left( \sum_{y\in Y\backslash \tY}\|y-\Delta\|^2\right)^{1/2} \right)^2\\
 & \leq  
 (2U+V+1)\left(  \sum_{x \in \tX} \|x-\phi(x)\| + \sum_{y\in Y\backslash \tY}\|y-\Delta\|^2 + \sum_{x\in X\backslash \tX}\|x-\Delta\|^2 + \overrightarrow {d_2}(\tX,\tY)^2 \right)  .
\end{align*}
Finally, we combine this with the fact that 
\begin{equation*}
 \sum_{y\in Y\backslash \tY}\|y-\Delta\|^2 + \sum_{x\in X\backslash \tX}\|x-\Delta\|^2 + \overrightarrow {d_2}(\tX,\tY)^2 = \overrightarrow {d_2}(X,Y)^2
\end{equation*}
to get the bound in the theorem.

\end{proof}


\section{Examples}
\label{sec:Examples}

\begin{figure}[h!]
        \centering
        \begin{subfigure}[b]{0.5\textwidth}
                \centering
                \includegraphics[width=\textwidth]{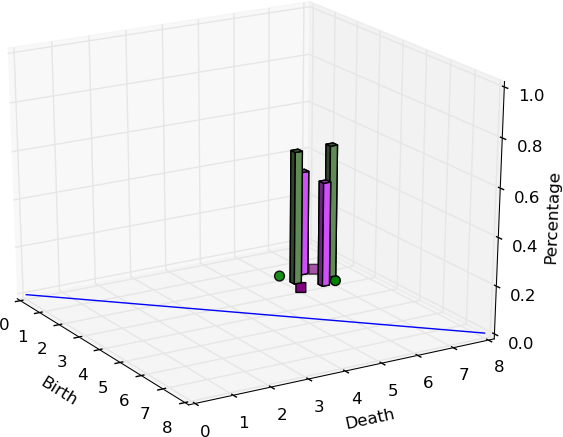}
                \label{fig:gull}
        \end{subfigure}%
        ~ 
        \begin{subfigure}[b]{0.5\textwidth}
                \centering
                \includegraphics[width=\textwidth]{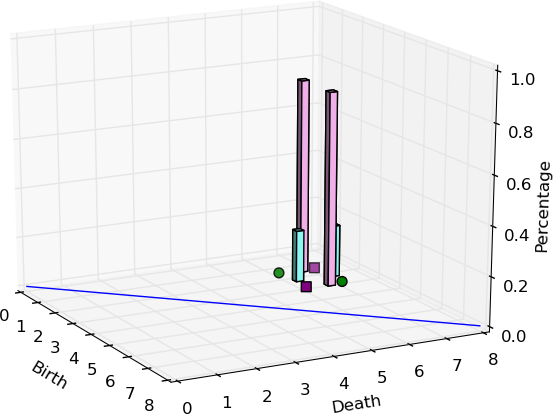}
                \label{fig:tiger}
        \end{subfigure}
        \caption{Change in mean distribution as a set of two diagrams moves through the problematic  configuration of Fig.~\ref{F:SquareExampleVineyard}. On the left we see the mean of two diagrams which form a rectangle that is slightly longer in the death-axis direction  On the right is the result for a rectangle that is quite a bit longer in the birth-axis direction.}
        \label{F:SquareColor}
\end{figure}

We now give some examples of the probabilistic \Frechet mean of a set of diagrams, introducing a useful way to visualize them along the way.
Recall that the mean distribution of a set of diagrams is a weighted sum of delta-measures, each one concentrated on the mean of one of the possible groupings among the diagrams, with the weights given by the probability that a perturbation of the diagrams would produce that grouping.
 
In Fig.~\ref{F:SquareColor}, we show a resolution to the discontinuity issue raised in Fig.~\ref{F:SquareExampleVineyard}, although this figure needs some explanation.
The flat colored dots on the left side of the figure represent a pair of diagrams which form a rectangle that is slightly longer in the death-axis direction. 
To approximate the probability of each possible matching between the pair, we perturbed the diagrams $100$ times with $\alpha = 0.3$
and $\eta_0$ equal to the uniform distribution, and simply counted the number of times each matching occurred.
The results are shown on the left side of the figure, where the height of a colored stack represents the weight of the diagram which contains the point at the bottom of the stack; note that the green stacks are slightly taller than the purple ones.
On the other hand, the right side of the same figure shows the mean distribution for a pair of diagrams which forms a rectangle that is quite a bit longer in the birth-axis direction.

For a more complicated example, we drew thirty different point clouds from a pair of linked annuli of different radii; one such point cloud is shown on the left of Fig.~\ref{F:TwoAnnulus}.
Then we computed the one-dimensional persistence diagram for each point cloud, using the recently-developed $M_{12}$ software package
\cite{deckard2013}.
As one might expect, each diagram contained a point for the big annulus, and point for the small annulus, and a good bit of noise along the diagonal. However, the birth times of the non-noisy points varied quite widely.
The set of thirty diagrams, overlaid in one picture, is shown on the right of the same figure.

\begin{figure}[bth]
        \centering
        \begin{subfigure}[b]{0.5\textwidth}
                \centering
                \includegraphics[width=\textwidth]{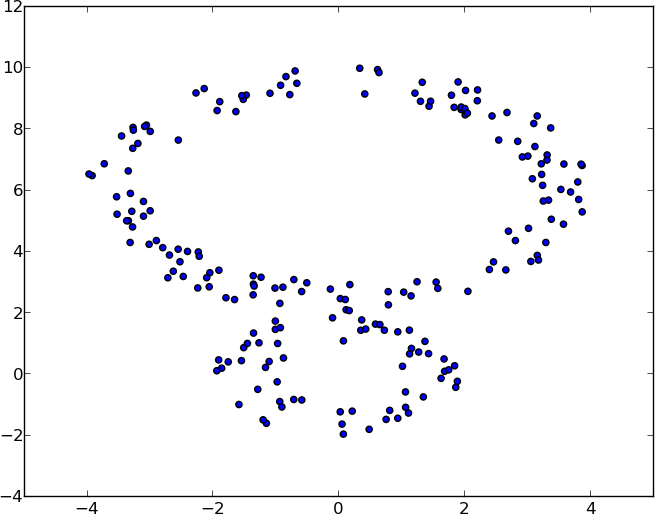}
                \label{fig:gull}
        \end{subfigure}%
        ~ 
        \begin{subfigure}[b]{0.5\textwidth}
                \centering
                \includegraphics[width=\textwidth]{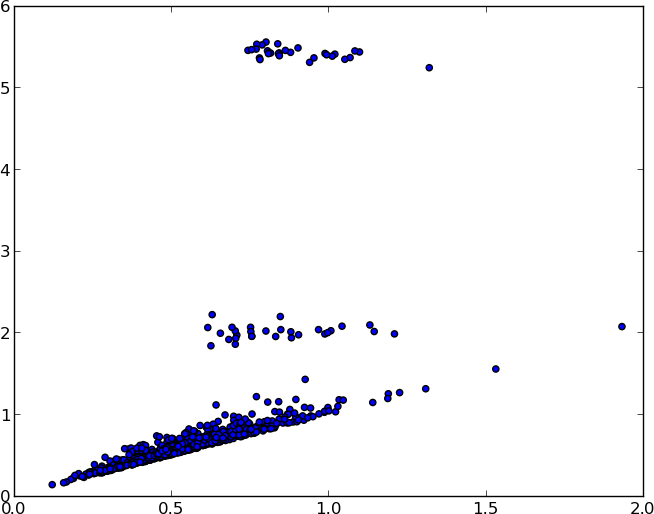}
                \label{fig:tiger}
        \end{subfigure}
        \caption{Thirty diagrams were created from thirty point clouds drawn from a double annulus. One such point cloud is shown on the left. All thirty diagrams are overlaid on the right.}\label{F:TwoAnnulus}
\end{figure}

Finally, we computed the mean distribution of these thirty diagrams, using the same approximation scheme as above.
On the left of Fig.~\ref{F:DistTwoAnnulus}, we see an overlay of the set of all diagrams which receive positive weight in the mean distribution, while the right side of the same figure displays the mean distribution using the same colored-stack scheme as in the example above.
Notice that the two very large stacks are actually at height one, which indicates that every single diagram in the mean contains the two non-noisy dots from the left side of the figure.

\begin{figure}[bth]
        \centering
        \begin{subfigure}[b]{0.5\textwidth}
                \centering
                \includegraphics[width=\textwidth]{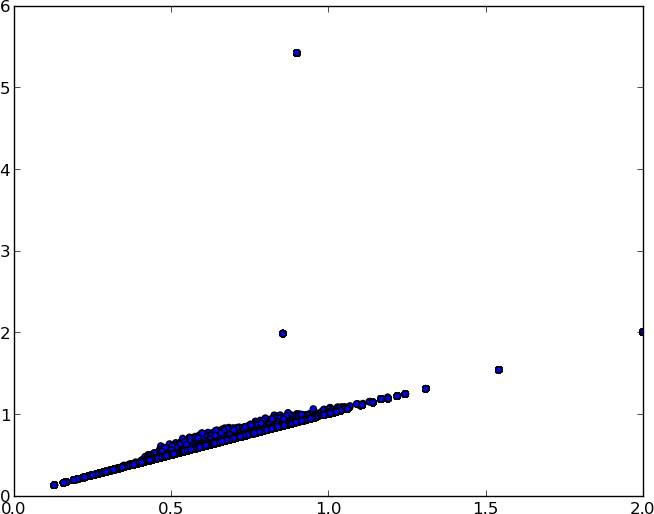}
                \label{fig:gull}
        \end{subfigure}%
        ~ 
        \begin{subfigure}[b]{0.5\textwidth}
                \centering
                \includegraphics[width=\textwidth]{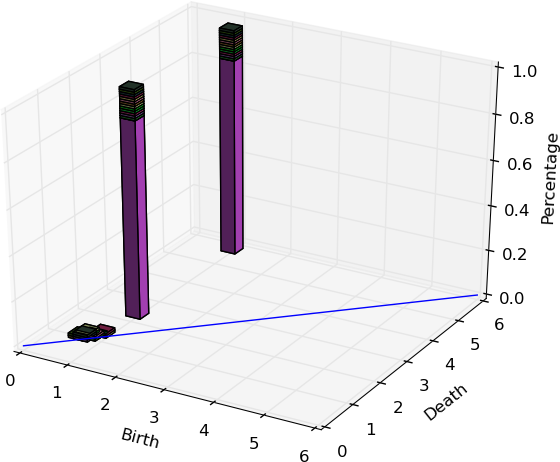}
                \label{fig:tiger}
        \end{subfigure}
        \caption{The mean distribution for a set of thirty diagrams sampled from a double annulus. The left side shows all positive-weight diagrams in the mean overlaid in one figure, while the right side indicates the weights in a three-dimensional plot.}\label{F:DistTwoAnnulus}
\end{figure}

\section{Conclusions  and Future Work}
\label{sec:Conclusion}

In this paper, we have defined a new mean which, unlike its predecessor, is continuous for continuously varying diagrams.
This mean is, in fact, a distribution on diagram space which is one feature of the distribution of diagrams from which it arose.
We hope that this new definition will provide a useful statistical tool for topological data analysis.
We also believe that this is an important step in the overall project of establishing persistent
homology as an important shape statistic.
Several questions remain, however, and there are obviously many directions for future research.
We list some of them here.

The most pressing need, of course, is to study how far we can take this new definition into the realm of traditional statistics. In particular, can we prove laws of large numbers, central limit theorems, and the like?
Will this mean actually provide a useful tool towards the bootstrapping idea discussed in the introduction?
Can we use this new mean, and the associated variance function, to provide more insight into the convergence rate theorems of \cite{chazal2013}?

On a more technical level, can we improve our continuity theorem to remove the reliance on the subspaces
$S_{M,K}$? At the moment, we can not find counterexamples to a more general statement, but nor can we prove the theorem without making finiteness assumptions.
We also conjecture that the constant can be improved.  
In particular, we are making a vast over-estimate by using $\overline M$.  
It would also be interesting to understand exactly how large a role $\alpha$ plays. 
Of course, if $\alpha$ goes to 0, the PFM should converge to the regular \Frechet mean, so can we  make a good choice of $\alpha$ based on the diagrams of study?

Note, too, that we have only addressed means and variances in this paper.
Another interesting statistical summary of data is the median; this will be addressed
in an upcoming paper \cite{Turner2013}.
Perhaps the most important project is to understand under what conditions persistence diagrams provide sufficient statistics for an object, a data cloud, etc.
The work in this paper will be a critical part of this effort.

\begin{figure}
 \centering
 \begin{subfigure}[b]{.4\textwidth}
  \includegraphics[width = \textwidth]{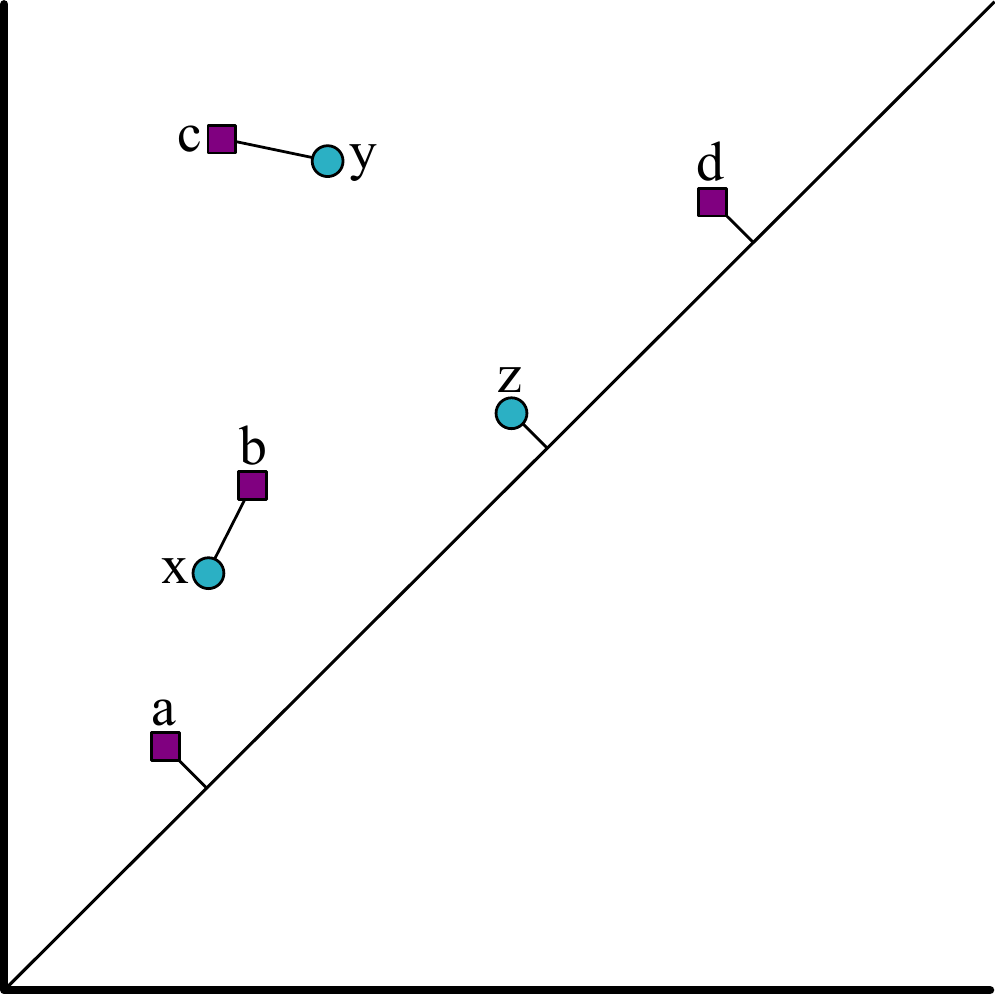}
  \caption{}
  \label{F:WassComputationA}
 \end{subfigure}
 \qquad
 \begin{subfigure}[b]{.4\textwidth}
  \includegraphics[width = \textwidth]{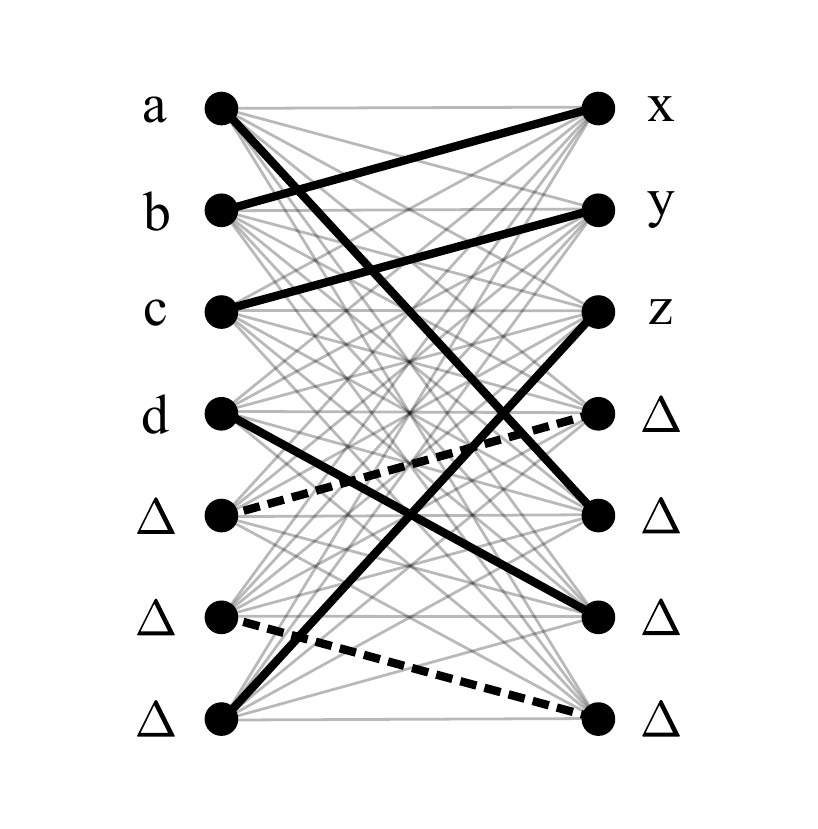}
  \caption{}
  \label{F:WassComputationB}
 \end{subfigure}
 \caption[Computation of the Wasserstein Distance]{Computation of the Wasserstein distance between $d_\blacksquare$ and $d_{\bullet}$ in Fig.~(a).  The problem is turned into the problem of computing a minimum cost grouping on the weighted graph in Fig.~(b).  The grouping chosen, shown in the bold edges in (b), is used to determine the matching for the diagrams in (a).  Dashed edges in (b) correspond to $\Delta-\Delta$ pairings, which contribute nothing to the total distance.  }
\label{F:WassComputation}
\end{figure}

\bibliographystyle{plain}
\bibliography{FuzzyMeans} 

 \appendix
 \section{Algorithms}
 \label{Appendix:Algorithms}

\begin{algorithm}    
\caption{Algorithm for computing the \Frechet Mean of a finite set of diagrams }     
\label{Alg:KatesMean}
\begin{algorithmic}  
    \Require Persistence diagrams $X_1,\cdots,X_N$
    \Ensure $Y$, a persistence diagram giving a local min of the \Frechet function
    
    \State Choose one of the $X_i$ randomly, set $Y = X_i$
    \State Initialize matching $G$
    \Comment $G[j,i]=$ the $x_k \in X_i$ matched \\
    \Comment with the point $y_j \in Y$
    
    \State stop $=$ \tt{False}
    \While {stop $==$ \texttt{False}}\\
    
	\For{ each diagram $X_i$} \Comment Determine the best $G$
	    \State $P = $WassersteinPairing($Y,X_i$)
	    \For{ each pair $(y_j,x_k) \in P$}
	      \State Set $G[j,i] = x_k$
	    \EndFor
	\EndFor\\
	
	\State Initialize empty diagram $Y'$
		    \Comment Move each point to the
	\For{ each point $y_j \in Y$}
		    \Comment  barycenter of its selection.
	    \State $y'_j = \textrm{mean}\{G[j,1],\cdots,G[j,N] \}$
		    \Comment $Y' = \textrm{mean}_X(G)$
	    \State Add $y'_j$ to $Y'$
	\EndFor\\
	
	\If{ WassersteinPairing$(Y,X_i) =$ WassersteinPairing$(Y',X_i)$ $\forall i$}
	    \State stop $=$ \texttt{True} 
	\EndIf
	\State $Y = Y'$
    \EndWhile\\
    \Return $Y$
\end{algorithmic}
\end{algorithm}

Here, we discuss the algorithm to compute an estimate of the \Frechet mean of a set of diagrams as given in \cite{Turner2011} using the vocabulary developed in this paper.
It is shown there that the \Frechet function is semiconcave for distributions with bounded support, so we can make use of a gradient descent algorithm to find local minima of the \Frechet function, Def.~\ref{D: Frechet}. 
In order to present the algorithm for computing the \Frechet mean, we must first describe the algorithm for computation of Wasserstein distance. 
In order to compute the Wasserstein distance between two diagrams, we will reduce the problem to computing a minimum cost grouping of a complete, weighted bipartite graph.

Let $X = [x_1,\cdots,x_k]$ and $Y = [y_1,\cdots,y_m]$ be diagrams.
In order to compute $W_2[L_2](X,Y)$, we construct a complete bipartite graph with vertex set $U \cup V$.  
There is a vertex in $U$ for each $x_i$, as well as $m$ vertices representing the abstract diagonal $\Delta$; similarly, $V$ has a vertex for each $y_i$ as well as $k$ vertices representing $\Delta$.  
The edge between points $x_i$ and $y_j$ is given weight $\|x_i - y_j\|^p$.
Each edge $(x_i,\Delta)$ and $(\Delta,y_j)$ has weight $\|x_i-\Delta\|^p$ and $\|y_j-\Delta\|^p$ respectively where $\|a-\Delta\| = \min_{z \in \Delta} \|a-z\|$.
Finally, edges between two vertices representing $\Delta$ are given weight 0. 
The minimum cost grouping algorithm typically used is the Hungarian algorithm of Munkres \cite{Munkres1957}. 

A minimum cost grouping in the bipartite graph immediately gives a matching $\phi:U\to V$ and the Wasserstein distance is given by the square root of the sum of the squares of the weights of the edges. 
Notice that since there could be multiple groupings for a bipartite graph which minimize the cost,  there could be multiple groupings which minimize the Wasserstein distance.
To compute the mean diagram, we will actually be more interested in the matching returned in this  algorithm than in the distance itself. 
Fig.~\ref{F:WassComputation} displays an example of a pair of  diagrams and their corresponding bipartite graph.

Now we are ready to give the algorithm for the \Frechet mean of a set of diagrams.  
Given a finite set of diagrams $\{X_1,\cdots,X_N\}$, start with a candidate for the mean, $Y$, and compute the matching for $W_2(Y,X_i)$.  
We denote this as WassersteinPairing$(Y,X_i)$.
From this, we have a grouping $G$ where $G[j,i]$ gives the point in $X_i$ which was paired to point $y_j \in Y$.
Set $Y' = \textrm{mean}(G)$.    
This new diagram is now the candidate for the mean and the process is repeated. 
The algorithm terminates when the Wasserstein pairing does not change.
In \cite{Turner2011}, the structure of $(D_{2},W_2)$ is used to prove that this algorithm terminates at a local minimum of the \Frechet function.
See Algorithm \ref{Alg:KatesMean} for the pseudocode.

\end {document}